\documentclass[12pt]{amsart}
\usepackage{amsmath,amsthm,amssymb}
\usepackage{graphicx}
\usepackage[margin=1in]{geometry}
\usepackage
[bookmarks,colorlinks,linkcolor=blue,citecolor=blue,urlcolor=blue]
{hyperref}
\begin{document}


\def\s{\Sym}
\def\w{\mathcal W}
\def\m{\mathcal M}
\def\P{\mathcal P}
\def\y{\mathcal Y}
\def\q{\mathcal Q}
\def\ssym{\mathfrak SSym}
\def\msym{\mathcal MSym}
\def\wsym{\mathcal WSym}
\def\psym{\mathcal PSym}
\def\ysym{\mathcal YSym}
\def\qsym{\mathcal QSym}
\def\dsym{\Delta Sym}
\def\Sym{\mathfrak S}
\def\bs{\boldsymbol}

\newcommand{\PGL} {\Pj\Gl_2(\R)}                           
\newcommand{\PGLC} {\Pj\Gl_2(\C)}                          
\newcommand{\RP} {\R\Pj^1}                                 
\newcommand{\CP} {\C\Pj^1}                                 
\newcommand{\C} {{\mathbb C}}                              
\newcommand{\R} {{\mathbb R}}                              
\newcommand{\Q} {{\mathbb Q}}                              
\newcommand{\Z} {{\mathbb Z}}                              
\newcommand{\Pj} {{\mathbb P}}                             
\newcommand{\Sg} {{\mathbb S}}                             
\newcommand{\Gl} {{\rm Gl}}                                

\newcommand{\suchthat} {\:\: | \:\:}
\newcommand{\ore} {\ \ {\it or} \ \ }
\newcommand{\oand} {\ \ {\it and} \ \ }

\newcommand{\oM} [1] {\ensuremath{{\mathcal M}_{0,#1}(\R)}}                 
\newcommand{\M} [1] {\ensuremath{{\overline{\mathcal M}}{_{0, #1}(\R)}}}    
\newcommand{\cM} [1] {\ensuremath{{\mathcal M}_{0,#1}}}                     
\newcommand{\CM} [1] {\ensuremath{{\overline{\mathcal M}}{_{0, #1}}}}       

\newcommand{\Tubeset} {\mathfrak{T}}                                        

\newcommand{\D} {\Delta}
\newcommand{\K} {\mathcal{K}}                           
\newcommand{\J} {\mathcal{J}}

\newcommand{\sm}{\varepsilon}

\newcommand{\rec}[1] {G^*(#1)}                  

\newcommand{\KG} {{\K} G}
\newcommand{\JG} {{\J} G}                

\newcommand{\temp} {\nabla}
\newcommand{\JGr} {\JG_r}
\newcommand{\JGd} {\JG_d}

%
%
\theoremstyle{plain}
\newtheorem{thm}{Theorem}
\newtheorem{prop}[thm]{Proposition}
\newtheorem{lem}[thm]{Lemma}
\newtheorem{cor}[thm]{Corollary}
\newtheorem{ques}[thm]{Question}
\newtheorem{conj}[thm]{Conjecture}
\newtheorem{defn}[thm]{Definition}
\newtheorem{exmp}[thm]{Example}
\newtheorem{rem}[thm]{Remark}
\newenvironment{defi}[1][]{\rm\begin{defn}[#1]\rm}{\end{defn}}
\newenvironment{exam}[1][]{\rm\begin{exmp}[#1]\rm}{\end{exmp}}
\newenvironment{rema}[1][]{\rm\begin{rem}[#1]\rm}{\end{rem}}
\numberwithin{thm}{section}

\addtolength{\textheight}{48pt}

%
%

\title {Geometric combinatorial algebras: cyclohedron and simplex}


\author{Stefan Forcey}
\address{S.\ Forcey: Tennessee State University, Nashville, TN 37209}
\email{sforcey@tnstate.edu}

\author{Derriell Springfield}

\begin{abstract}
In this paper we report on results of our investigation into the
algebraic structure supported by the combinatorial geometry of the
cyclohedron.
 Our new graded algebra structures
lie between two well known Hopf algebras: the Malvenuto-Reutenauer
algebra of permutations and the Loday-Ronco algebra of binary trees.
 Connecting algebra
maps arise from a new generalization of the Tonks
 projection from the permutohedron
to the associahedron, which we discover via the viewpoint of the
graph associahedra of Carr and Devadoss. At the same time, that
viewpoint allows exciting geometrical insights into the
multiplicative structure of the algebras involved. Extending the
Tonks projection also reveals a new graded algebra structure on the
simplices. Finally this latter is extended to a new graded Hopf
algebra (one-sided) with basis all the faces of the simplices.
\end{abstract}


\maketitle


\baselineskip=15pt

%
%
\section{Introduction}
\subsection{Background: Polytope algebras} In 1998 Loday and Ronco found an intriguing Hopf algebra of planar binary
trees lying between the Malvenuto-Reutenauer Hopf algebra of
 permutations \cite{MR} and the Solomon descent algebra of Boolean subsets \cite{LR}.
They also described natural Hopf algebra maps which neatly factor the descent map from permutations to Boolean subsets.
Their first factor turns out to be the restriction (to vertices) of the Tonks projection from the
 permutohedron to the associahedron. Chapoton made sense of this latter fact when he
found larger Hopf algebras based on the faces of the respective polytopes \cite{chap}. Here we study several
new algebraic structures based on polytope sequences, including the cyclohedra, $\w_n,$ and the simplices,
$\Delta^n$. In Figure~\ref{f:cyc1} we show the central polytopes, in three dimensions.

\begin{figure}[h]
\includegraphics[width=\textwidth]{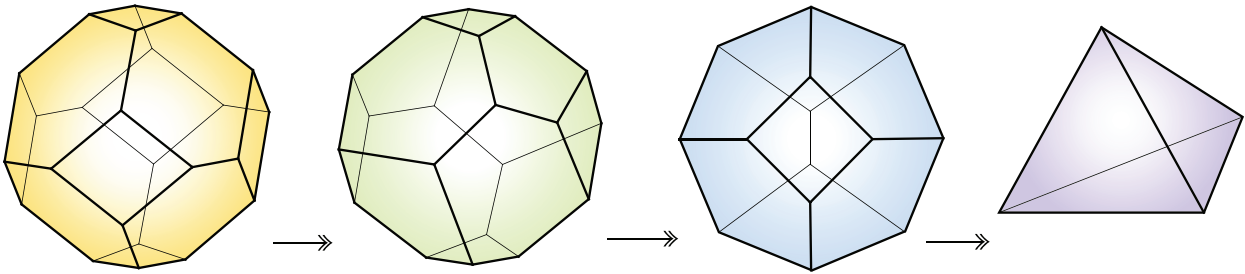}
\caption{The main characters, left to right: $\P_4, \w_4,$ $\K_4,$ and $\Delta^3.$ } \label{f:cyc1}
\end{figure}

We will be referring to the algebras and maps discussed
in \cite{AS-LR} and \cite{AS-MR}. In these two papers, Aguiar and Sottile make powerful
 use of the weak order on the symmetric groups and the Tamari order on binary trees. By leveraging the M\"obius
function of these two lattices they provide clear descriptions of
the antipodes and of the geometric underpinnings of the Hopf
algebras. The also demonstrate cofreeness, characterize primitives
of the coalgebras, and, in the case of binary trees, demonstrate
equivalence to the non-commutative Connes-Kreimer Hopf algebra from
renormalization theory.

Here we generalize the well known algebras based on
  associahedra and permutohedra to new ones on cyclohedra and
simplices. The cyclohedra underlie graded algebras, and the
simplices underlie  a new (one-sided) Hopf algebra. We leave for
future investigation many potential algebras and coalgebras based on
novel sequences of graph associahedra. The phenomenon of polytopes
underlying Hopf structure may be rare, but algebras and coalgebras
based on polytope faces are beginning to seem ubiquitous.
 There does exist a larger family of Hopf algebras to be uncovered in the structure
of the polytope sequences derived from trees, including the multiplihedra, their quotients and their covers.
These are studied in \cite{newFLS}.

\subsection{Background: Cyclohedra} The cyclohedron $\mathcal{W}_n$ of dimension $n-1$ for $n$ a positive
integer was originally
 developed  by Bott and Taubes \cite{bott},
and received its name from Stasheff. The name points out the close
connection to Stasheff's associahedra, which we denote $\K_n$ of
dimension $n-1.$ The former authors
described the facets of $\mathcal{W}_n$ as being indexed by subsets
of $[n] = {1,2,\dots,n}$ of
 cardinality $\ge 2$ in cyclic
order. Thus there are $n(n-1)$ facets. All the faces can be indexed
by cyclic bracketings of the string $123\dots n$, where the facets
have exactly one pair of brackets. The vertices are complete
bracketings, enumerated by ${2(n-1) \choose n-1}.$

\begin{figure}[h]
\includegraphics[width=\textwidth]{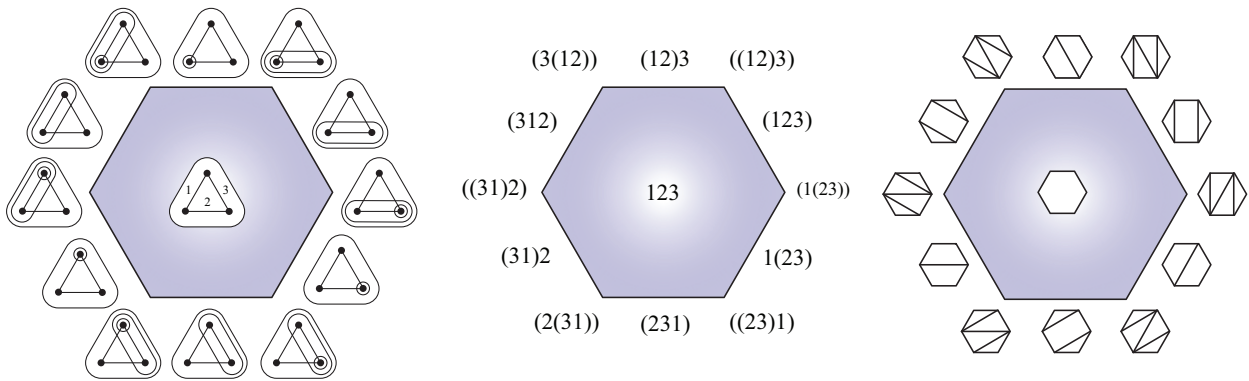}
\caption{The cyclohedron $\w_3$ with various indexing.}
\label{f:indexing}
\end{figure}

 The space $\mathcal{W}_n \times S^1$, seen as the
  compactification of the
configuration space of $n$ distinct points in $\R^3$ which are
constrained to lie upon a given knot, is used to define new
invariants which reflect the self linking of knots \cite{bott}.
Since their inception the cyclohedra have proven to be useful in
many other arenas. They provide an excellent example of a right
operad module
 (over the operad of associahedra) as shown in \cite{markl}. Devadoss discovered a
 tiling of the $(n-1)$-torus by $(n-1)!$ copies
 of $\mathcal{W}_n$ in  \cite{devspace}.
 Recently the cyclohedra have
  been used to look for the statistical signature of periodically
 expressed genes in the study of biological clocks \cite{morton}.

The faces of the cyclohedra may also be indexed by the centrally
symmetric subdivisions of a convex polygon of $2n$ sides, as
discovered by Simion \cite{simion}. In this indexing the vertices
are centrally symmetric triangulations, which allowed Hohlweg and
Lange to develop geometric realizations of the cyclohedra as convex
hulls \cite{hohl_lang}. This picture is related to the work of
 Fomin, Reading and Zelevinsky, who see the cyclohedra
 as a generalization of the associahedra corresponding to the $B_n$ Coxeter diagrams \cite{reading-camb}.
  From their perspective the face structure of the cyclohedron is
 determined by
 the sub-cluster structure of the generators of a finite cluster algebra.

In contrast, for Devadoss the cyclohedra arise from truncating simplex faces
 corresponding to sub-diagrams of the $\tilde{A}_n$ Coxeter diagram, or cycle graph \cite{dev-carr}.
     In this paper we will work from the point
 of view taken by Devadoss and consider the faces as indexed by tubings of the cycle graph on $n$ vertices.
   Given a graph $G$, the
  \emph{graph associahedron} $\KG$ is a convex polytope generalizing the associahedron,
   with a face poset based on the full connected subgraphs, or tubes of $G$.  For instance,
   when $G$ is a path, a cycle, or a complete graph, $\KG$ results in the associahedron,
   cyclohedron, and permutohedron, respectively.  In \cite{dev-real}, a geometric realization
    of $\KG$ is given, constructing this polytope from truncations of the simplex.  In \cite{dev-carr} the motivation
for the development of $\KG$ is that it appears in tilings of minimal blow-ups of certain Coxeter complexes,
      which themselves are natural generalizations of the moduli spaces \M{n}.

The value of the graph associahedron to  algebraists, as we hope to
demonstrate here, is twofold. First, a unified description of so
many combinatorial polytopes allows useful generalizations of the
known algebraic structures on familiar
 polytope sequences.
 Second, the recursive structure
 described by Carr and
Devadoss in a general way for graph associahedra
 turns out to lend new geometrical meaning to the graded algebra structures of both the Malvenuto-Reutenauer and Loday-Ronco
algebras.  The product of two vertices, from terms $P_i$ and $P_j$ of
 a given sequence of polytopes $\{P_n\}_{n=0}^{\infty}$, is described as a sum of vertices of the term $P_{i+j}$ to which the operands are
 mapped.  The
summed vertices in the product
 are the images of classical
inclusion maps composed with our new extensions of the Tonks
projection.
\subsubsection{Acknowledgements}
We would like to thank the referees for helping us make connections to other work, and for making some
excellent suggestions about presentation of the main ideas. We also thank the following for tutorials and
helpful conversations on the subject matter: Satyan Devadoss, Aaron Lauve, Maria Ronco and Frank Sottile.

\subsection{Summary}
\subsubsection{Notation}
The Hopf algebras of permutations, binary trees, and Boolean subsets
are denoted respectively $\ssym, \ysym $ and $\qsym,$ as in
\cite{AS-LR} and \cite{AS-MR}.  Note that some of our sources,
including Loday and Ronco's original treatment of binary trees,
actually deal with the dual graded algebras. The larger algebras of
faces of permutohedra, associahedra and cubes are denoted
$\tilde{\ssym}, \tilde{\ysym} $ and $\tilde{\qsym}.$The new algebras
of cyclohedra vertices and faces are denoted $\wsym$ and
$\tilde{\wsym}.$ The new algebra of vertices of the simplices is
denoted $\dsym.$ Finally the new one-sided Hopf algebra of faces of
the simplices is denoted $\tilde{\dsym}.$ Throughout there are three
important maps we will use. First we see them as polytope maps:
$\hat{\rho}$ is the inclusion of facets defined by Devadoss as a
generalization of the recursive definitions of associahedra,
simplices and permutohedra; $\eta$ is  the projection from the
polytope of a disconnected graph to its components; and $\Theta$ is
the generalization of Tonks's projection from the permutohedron to
the associahedron. This last gives rise to maps $\hat{\Theta}$ which
are algebra homomorphisms.

\subsubsection{Main Results} In Theorem~\ref{t:biggy} we demonstrate
that $\wsym$ is an associative graded algebra. In
Theorem~\ref{t:wface} we extend this structure to the full poset of
faces to describe the associative graded algebra $\tilde{\wsym}.$ In
Theorem~\ref{t:sim} we demonstrate an associative  graded algebra
structure on $\dsym.$ In Theorem~\ref{t:hopf} we show how to extend
this structure to become a new graded (one-sided) Hopf algebra
$\tilde{\dsym},$ based upon all the faces of the simplices. Thus its
graded dimension is $2^n$.

Theorems~\ref{t:geo_pro}, \ref{r:geo} and~\ref{t:geochap} and Remark~\ref{r:geo_cyc} point out that the
multiplications in all the algebras studied here can be understood in a unified way
 via the recursive structure of the
face posets of associated polytopes. To summarize, the products
 can be seen as a process of projecting and including of faces, via $\Theta,\hat{\rho}$ and $\eta$.
The algebras based on the permutohedron don't use $\Theta$ or $\eta,$ and the algebras based on the simplices
don't use $\Theta,$ but otherwise the definitions will seem quite repetitive. We treat each definition
separately to highlight those differences, but often the proofs of an early version will be referenced rather
than repeated verbatim.
 The coproducts of $\ysym$ and
$\ssym$ are also understandable as projections of the polytope geometry, as
mentioned in Remarks~\ref{r:geo_cop} and~\ref{r:y_cop}.

Before discussing algebraic structures,
however, we build a geometric combinatorial foundation. We show precisely how the graph associahedra are all
cellular quotients of the permutohedra (Lemma~\ref{l:newtonks}), and how the associahedron is a quotient of
any given connected graph associahedron (Theorem~\ref{t:ftonks}). Results similar to these latter statements
are implied by the work of Postnikov in \cite{post}, and were also reportedly known to Tonks in the cases
involving the cyclohedron \cite{dev-carr}.

The various maps arising from factors of the Tonks projection are shown to be algebra homomorphisms in
Theorems~\ref{t:permhom}, ~\ref{t:whom} and Lemma~\ref{l:next}. Corollary~\ref{c:cool} points out the
implication of  module structures on $\wsym$ over $\ssym$; and on $\ysym$ over $\wsym.$

\subsubsection{Overview of subsequent sections}
 Section~\ref{s:defns} describes the posets of connected subgraphs which are realized as the graph associahedra polytopes.
We also review the cartesian product structure of their facets.
Section~\ref{s:tonks} shows how the Tonks cellular projection
from the permutohedron to the associahedron can be factored in
 multiple ways so that the intermediate polytopes are graph associahedra
for connected graphs. We also show that the Tonks projection can be
extended to cellular projections to the simplices,
 again in such a way that
the intermediate polytopes are graph associahedra. In particular, we focus on
 a factorization of the Tonks projection which
has the cyclohedron (cycle graph associahedra) as an intermediate quotient polytope
 between permutohedron and associahedron.
In Section~\ref{s:newview} we use the viewpoint of graph
associahedra to redescribe the products in $\ssym$ and $\ysym$, and
to point out their geometric interpretation. The latter is based upon our
new cellular projections as well as classical
 inclusions of (cartesian products of) low dimensional
associahedra and permutohedra as faces in the higher dimensional polytopes of their respective sequences.
In Section~\ref{s:cyclo} we begin our exploration of new graded algebras with the vertices of the
 cyclohedron. Then we show that
the linear projections following from the factored Tonks projection
(restricted to vertices) are algebra maps. We extend these findings
in Section~\ref{s:chap} to the full algebras based on all the faces
of the polytopes involved. Finally in Section~\ref{s:simp} we
generalize our discoveries to the sequence of edgeless graph
associahedra. This allows us to build a graded algebra based upon
the vertices of the simplices. By carefully extending that structure
to the faces of the simplices we find a new graded (one-sided) Hopf
algebra, with graded dimension $2^n$.

%
%
\section{Review of some geometric combinatorics} \label{s:defns}

We begin with definitions of graph associahedra; the reader is encouraged to
 see \cite[Section 1]{dev-carr} and \cite{dev-forc} for further details.

\begin{defi}
Let $G$ be a finite connected simple graph.  A \emph{tube} is a set of nodes of $G$ whose induced
graph is a connected subgraph of $G$.  Two tubes $u$ and $v$ may interact on the graph as follows:
\begin{enumerate}
\item Tubes are \emph{nested} if  $u \subset v$.
\item Tubes are \emph{far apart} if $u \cup v$ is not a tube in $G,$ that is, the induced subgraph of the union
 is not connected, or
none of the nodes of $u$ are adjacent to a node of $v$.
\end{enumerate}
Tubes are \emph{compatible} if they are either nested or far apart.
We call $G$ itself the \emph{universal tube}.
 A \emph{tubing} $U$ of $G$ is a set of tubes of $G$ such that every pair of tubes in $U$ is
 compatible; moreover, we force every tubing of $G$ to contain (by default) its universal tube.
By the
 term $k$-\emph{tubing} we refer to a tubing made up of $k$ tubes, for $k \in \{1,\dots,n\}.$
 \end{defi}

When $G$ is a disconnected graph with connected components $G_1$,
\ldots, $G_k$, an additional condition is needed: If $u_i$ is the
tube of $G$ whose induced graph is $G_i$, then any tubing of $G$
cannot contain all of the tubes $\{u_1, \ldots, u_k\}$. However, the
universal tube is still included despite being disconnected.
 Parts (a)-(c) of Figure~\ref{f:legaltubing} from \cite{dev-forc} show examples of allowable tubings,
 whereas (d)-(f) depict the forbidden ones.

\begin{figure}[h]
\includegraphics[width=\textwidth]{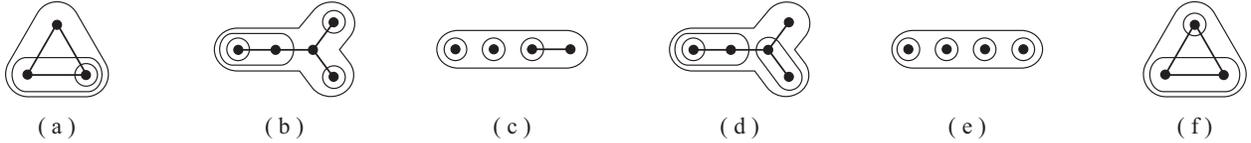}
\caption{(a)-(c) Allowable tubings and (d)-(f) forbidden tubings,
figure from \cite{dev-forc}.} \label{f:legaltubing}
\end{figure}


\begin{thm} {\textup{\cite[Section 3]{dev-carr}}} \label{t:graph}
For a graph $G$ with $n$ nodes, the \emph{graph associahedron} $\KG$
is a simple, convex polytope of dimension $n-1$ whose face poset is
isomorphic to the set of tubings of $G$, ordered by the relationship $U \prec
U'$ if $U$ is
 obtained from $U'$ by adding tubes.
\label{d:pg}
\end{thm}
 The vertices of the graph associahedron are the $n$-tubings of $G.$
 Faces of dimension $k$ are indexed by $(n-k)$-tubings of $G.$ In fact,
 the barycentric subdivision of ${\mathcal K}G$ is precisely the geometric realization of the described poset of tubings.
 Many of the face vectors of  graph associahedra for path-like graphs have been found, as shown in \cite{post2}.
  This source also contains the face vectors for the cyclohedra. There are many open questions
regarding formulas for the face vectors of
graph associahedra for specific types of graphs.

\begin{exam}
Figure~\ref{f:kwexmp}, partly from \cite{dev-forc}, shows two
examples of graph associahedra.
 These have underlying graphs a path and a disconnected graph, respectively, with three nodes each.
  These turn out to be the 2 dimensional associahedron and a square.
  The case of a three node complete graph, which is both the
  cyclohedron and the permutohedron, is shown in
  Figure~\ref{f:indexing}.
\end{exam}

\begin{figure}[h]
\includegraphics[width=\textwidth]{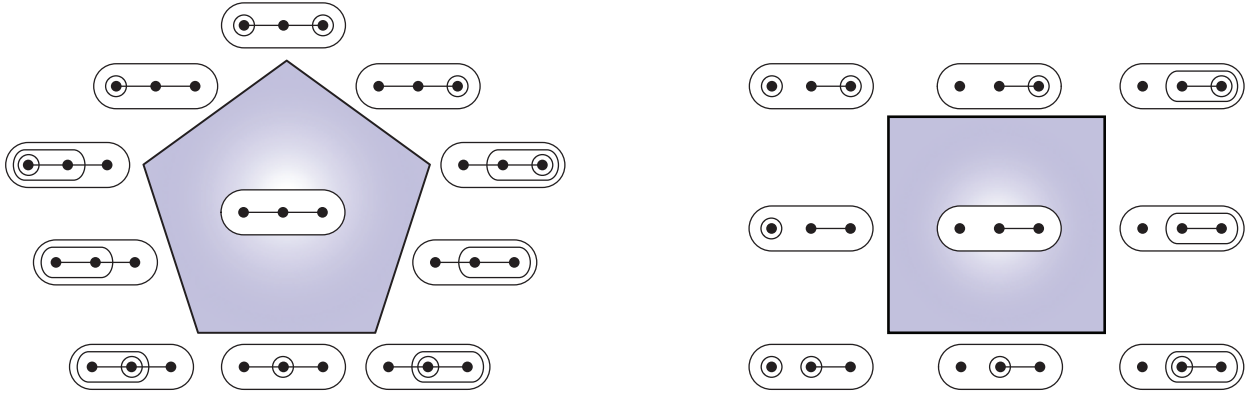}
\caption{Graph associahedra of a path and a disconnected graph. The 3-cube is found as the
graph-associahedron of two disjoint edges on four nodes, but no simple graph yields the 4-cube.}
\label{f:kwexmp}
\end{figure}

To describe the face structure of the graph associahedra we need a
definition from \cite[Section 2]{dev-carr}.

\begin{defi}
For graph $G$ and a collection of nodes $t$, construct a new graph
$\rec{t}$ called the \emph{reconnected complement}: If $V$ is the
set of nodes of $G$, then $V-t$ is the set of nodes of $\rec{t}$.
There is an edge between nodes $a$ and $b$ in $\rec{t}$ if $\{a,b\} \cup t'$ is connected in $G$ for some $t'\subseteq t$.
\end{defi}
\begin{exam}
\noindent Figure~\ref{f:recon} illustrates some examples of graphs
along with their reconnected complements.
\end{exam}

\begin{figure}[h]
\includegraphics {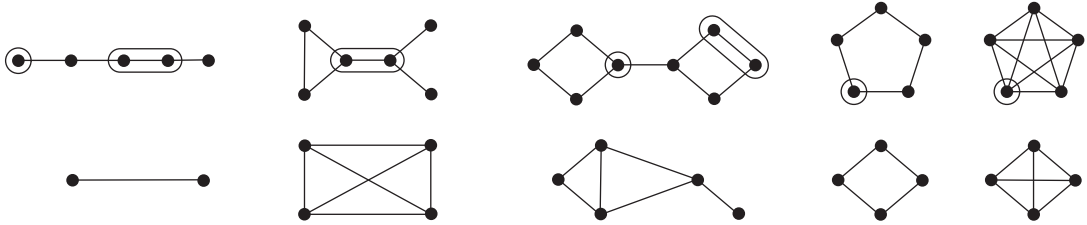}
\caption{Examples of tubes and their reconnected complements.}
\label{f:recon}
\end{figure}

For a given tube $t$ and a graph $G$, let $G(t)$ denote the induced
subgraph on the graph $G$. By abuse of notation, we sometimes refer
to $G(t)$ as a tube.
\begin{thm} \label{t:facet} \cite[Theorem 2.9]{dev-carr}
Let $V$ be a facet of $\KG,$ that is, a face of dimension $n-2$ of ${\mathcal K}G$, where $G$ has $n$ nodes.
Let $t$ be the single, non-universal, tube of $V$ . The face poset of $V$ is isomorphic to $ \KG(t) \times \K
\rec{t}$.
\end{thm}

A pair of examples is shown in Figure \ref{f:facet}.  The isomorphism
described in \cite[Theorem 2.9]{dev-carr} is called $\hat{\rho}.$
Since we will be using this isomorphism more than once as an
embedding  of faces in polytopes, then we will specify it, according
to the tube involved, as $\hat{\rho}_t:\KG(t) \times \K
\rec{t}\hookrightarrow\KG.$

Given a pair of tubings from $T\in \KG_t$ and $T'\in \K\rec{t}$ the image $\hat{\rho}(T,T')$ consists of all
of $T$ and an expanded version of $T'.$ In the latter each tube of $T'$ is expanded by taking its union with
$t$ if that union is itself a tube. A specific example of the action of $\hat{\rho}_t$ is in
Figure~\ref{f:eta_rho}. In fact, often there will be more than one tube involved, as we now indicate:

\begin{cor}\label{c:mfacet}
Let $\{t_1,\dots ,t_k,G\}$ be an explicit tubing of $G$, such that each pair of non-universal
tubes in the list is far apart.
Then the face of $\KG$ associated to  that tubing is isomorphic to
 $\KG(t_1)\times \dots \times \KG(t_n)\times \K\rec{t_1\cup\dots\cup t_k}.$
\end{cor}
We will denote the embedding by:
$$\hat{\rho}_{t_1\dots t_k}:\KG(t_1)\times \dots \times
\KG(t_n)\times \K\rec{t_1\cup\dots\cup t_k}\hookrightarrow\KG.$$
\begin{proof}
This follows directly from Theorem~\ref{t:facet}. Note that
the reconnected complement with respect to the union of several tubes
$t_1,\dots,t_k$ is the same as taking successive iterated reconnected complements with respect to each tube in the list.
That is, $$\rec{t_1\cup\dots\cup t_k} = (\dots((G^*(t_1))^*(t_2))^*\dots)^*(t_k).$$
\end{proof}

\begin{figure}[h]
\includegraphics[width=\textwidth]{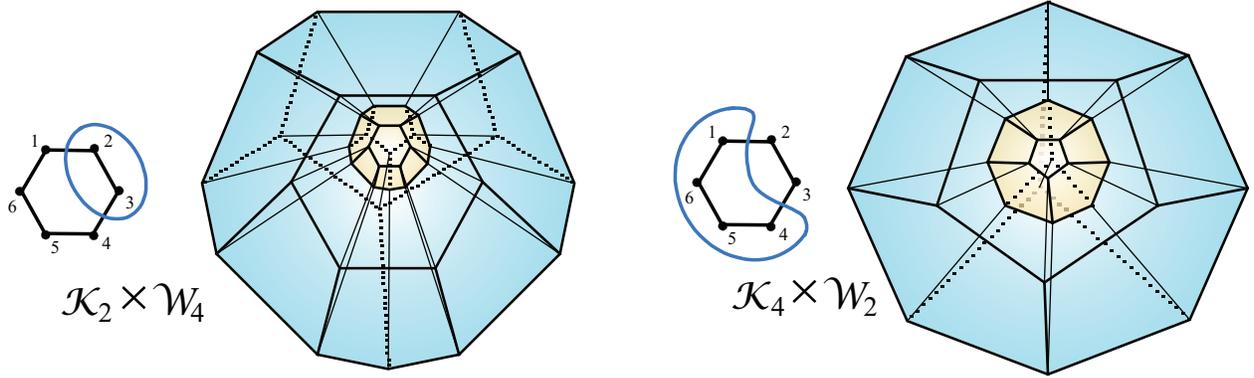}
\caption{Two facets of the cyclohedron $\w_6$.} \label{f:facet}
\end{figure}

\begin{rema}
As shown in \cite{dev-real}, the graph associahedron of the graph consisting of $n$ nodes and no edges is the
$(n-1)$-simplex $\Delta^{(n-1)}$. Thus the graph associahedron of a graph with connected components $G_1,
G_2,\dots,G_k$ is actually equivalent to the polytope $\KG_1 \times\dots\times\KG_k\times\Delta^{(k-1)}.$
\end{rema}
This equivalence implies that:
\begin{lem}\label{l:comp}
For a disconnected graph $G$ with multiple connected components
$G_1, G_2,\dots,G_k$, there is always a cellular surjection
$$ \eta: \KG \to \KG_1 \times\dots\times\KG_k.$$
\end{lem}
An example is in Figure~\ref{f:eta_rho}.

\begin{figure}[h]
                  \centering
                  \includegraphics[width=\textwidth]{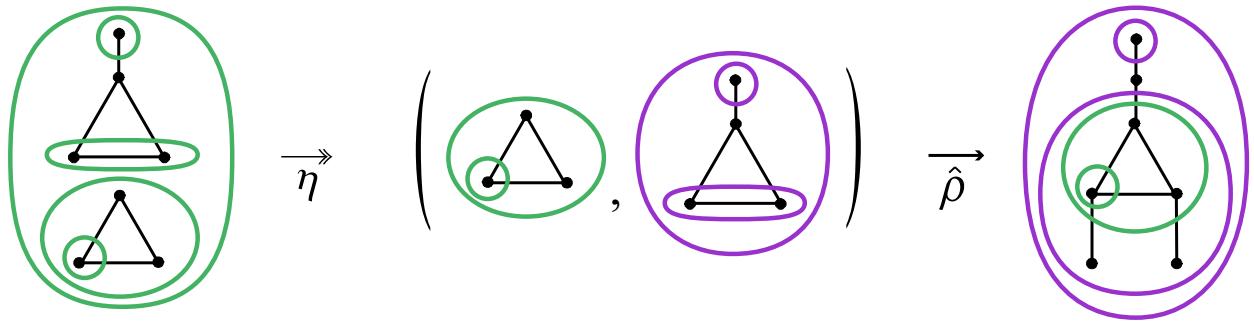}
     \caption{Example of the action of $\eta$ and $\hat{\rho}_t,$ where $t$ is the cycle sub-graph of the final graph.}\label{f:eta_rho}
  \end{figure}

%
%

\section{Factoring and extending the Tonks projection.}\label{s:tonks}
\subsection {Loday and Ronco's Hopf algebra map}
 The two most important existing mathematical structures we will use in this paper
  are the graded Hopf algebra
  of permutations, $\ssym$, and the graded Hopf algebra of binary trees, $\ysym$.
  The $n^{th}$ component of $\ssym$ has basis the
  symmetric group $S_n$, with number of elements counted by $n!$.  The $n^{th}$
  component of  $\ysym$ has basis the collection of binary trees with $n$
  interior nodes, and thus $n+1$ leaves, denoted $\y_n$. These
  are counted by the Catalan numbers.

  The connection discovered by Loday and Ronco between the two algebras
   is due to the
  fact that a permutation on $n$ elements can be pictured as a binary tree with $n$ interior nodes, drawn so
   that the interior nodes are at $n$ different vertical
  heights from the base of the tree. This is called an \emph{ordered binary tree}. The interior nodes are numbered left to right.
We number the leaves $0,1,\dots,n-1 $
and the nodes $1,2,\dots,n-1.$ The $i^{th}$ node is ``between'' leaf $i-1$ and leaf $i$ where
``between'' might be described to mean that a rain drop falling between those leaves would be
caught at that node. Distinct vertical levels of the nodes are numbered top to bottom.
   Then for a permutation $\sigma \in S_n$ the
  corresponding tree has node $i$ at level $\sigma(i).$ The map from permutations to binary trees is
  achieved by forgetting the levels, and just retaining the tree. This classical surjection is
 denoted $\tau: S_n \to \y_n$. An example is in
 Figure~\ref{perm1}.

By a \emph{cellular surjection} of polytopes $P$ and $Q$ we refer to a map $f$ from the face poset of $P$ to
that of $Q$ which is onto and which preserves the poset structure. That is, if $x$ is a sub-face of $y$ in
$P$ then $f(x)$ is a sub-face of or equal to $f(y).$ A \emph{cellular projection} is a cellular surjection
which also has the property that the dimension of $f(x)$ is less than or equal to the dimension of $x.$
 In \cite{tonks} Tonks
  extended $\tau$ to a cellular projection from the permutohedron to the associahedron: $\Theta: \P_n \to \K_n$.
In the cellular projection a face of the permutohedron, which is leveled tree, is taken to its underlying
tree, which is a face of the associahedron.
   Figure~\ref{perm2} shows an example.
The new revelation of Loday and Ronco is that the map $\tau$ gives
rise to a Hopf algebraic
   projection $\bs{\tau}:\ssym \to \ysym$, so that
   the algebra of binary trees is seen to be embedded in the algebra of permutations.

\begin{figure}[h]
                  \centering
                  \includegraphics[width=\textwidth]{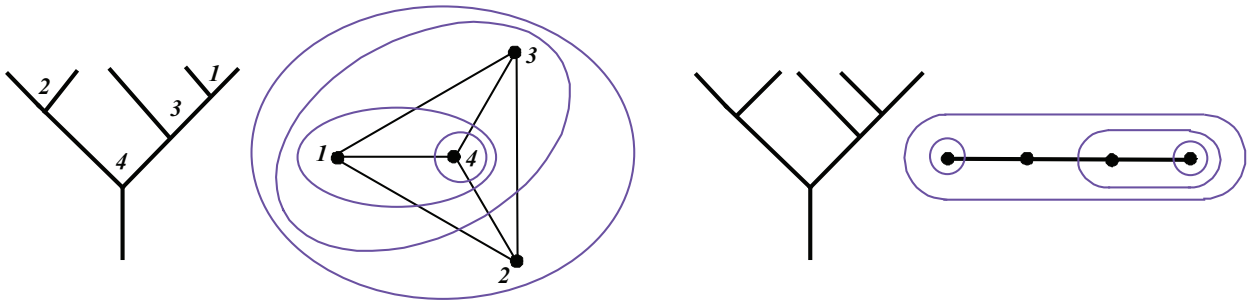}
     \caption{The permutation $\sigma = (2431) \in S_4$ pictured as an ordered tree and
                   as a tubing of the complete graph; the unordered binary tree, and its corresponding tubing.}\label{perm1}
  \end{figure}

%

\subsection{Tubings, permutations, and trees.}
   Our new approach to the Tonks projection is made possible by the recent discovery of
Devadoss in \cite{dev-real} that the
   graph-associahedron of the complete graph on $n$ vertices is precisely
   the $n^{th}$ permutohedron $\P_n$.
 Each of its vertices corresponds to a
   permutation of $n$ elements. Its faces in general correspond to ordered partitions of $[n].$ Keep in mind that for
a permutation $\sigma \in S_n,$ the corresponding ordered partition
of $[n]$ is
$(\{\sigma^{-1}(1)\},\{\sigma^{-1}(2)\},\dots\{\sigma^{-1}(n)\}).$

   Here is how to describe a bijection from the $n$-tubings on the complete graph to the permutations, as found by
Devadoss in \cite{dev-real}. First
   a numbering of the $n$ nodes must be chosen. Given an $n$-tubing,
   since the tubes are all nested we can number them
   starting with the innermost tube. Then the permutation $\sigma \in S_n$
   pictured by our $n$-tubing is such that node $i$
   is within tube $\sigma(i)$ but not within any tube nested inside of tube $\sigma(i).$
    Figure \ref{perm1} shows an example.

    It is easy to extend this bijection between $n$-tubings and permutations
    to all tubings of the complete graph and ordered partitions of $[n]$. Given a $k$-tubing of the complete graph,
    each tube contains some numbered nodes
   which are not contained in any other tube. These subsets of $[n],$ one for each tube, make up the $k$-partition, and the
   ordering of the partition is from innermost to outermost. Recall
   that an ordered $k$-partition of $[n]$ corresponds to a leveled
   tree with $n+1$ leaves and $k$ levels, numbered top to bottom, at which lie the internal nodes. Numbering the $n$ spaces between leaves
from left to right (imagine a raindrop falling into each space), we
see that the raindrops collecting at the internal nodes at level $i$
represent the $i^{th}$ subset in the partition.
   We denote our bijection by:
   $$
   f: \{ \text{leveled trees with $n$ leaves} \} \to \K( \text{complete graph on $n-1$ nodes}).
   $$
   Figure \ref{perm2} shows an example.

  \begin{figure}[h]
                  \centering
                  \includegraphics[width=\textwidth]{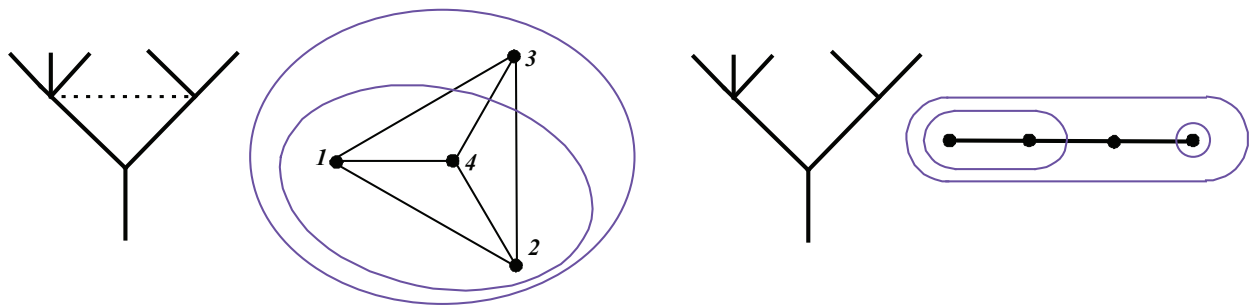}
     \caption{The ordered partition $(\{1,2,4\},\{3\})$ pictured as a leveled tree and
                   as a tubing of the complete graph; the underlying tree, and its corresponding tubing.}\label{perm2}
  \end{figure}

 The binary trees with $n+1$ leaves (and $n$ internal nodes)
 correspond to the vertices of the $(n-1)$-dimensional associahedron, or Stasheff polytope $\K_n$.
 In the world of graph-associahedra, these vertices correspond to the $n$-tubings of the path graph on $n$ nodes.
 Carr and Devadoss realized that in fact the path graph associahedron
 is precisely the Stasheff associahedron \cite{dev-carr}.
 Thus for any tree with $n+1$ leaves we can bijectively determine a tubing on the path graph.
 This correspondence can be described by creating a tube of the path graph for each internal node of the tree.
 We number the leaves from left to right $0,1,\dots,n$ and the nodes of the path from left to right $1,\dots,n.$
 The tube we create
 contains the same numbered nodes of the path graph as
  all but the leftmost leaf of the subtree determined by the internal node.
  This bijection we denote by:
  $$g: \{ \text{trees with $n$ leaves} \} \to \K( \text{path on $n-1$ nodes} ).$$
   Figures~\ref{perm1} and~\ref{perm2} show examples.

\subsection{Generalizing the Tonks projection}
The fact that every graph of $n$ nodes is a subgraph of the complete graph leads us to a grand factorization of the
Tonks cellular projection through all connected graph associahedra. Incomplete graphs are formed simply by removing edges
from the complete graph. As a single edge is deleted, the tubing
 is preserved up to connection. That is, if the nodes of a tube are no longer connected, it becomes
 two tubes made up of the two connected subgraphs spanned by its original set of nodes.
\begin{defi}
 Let $G$ be a graph on $n$ nodes, and let $e$ be an edge of $G$. Let $G-e$ denote the graph achieved be the deletion
of $e,$ while retaining all $n$ nodes.  We define a cellular projection
 $\Theta_e: \K G \twoheadrightarrow \K (G - e).$
First, allowing an abuse of notation, we define $\Theta_e$ on individual tubes. For $t$ a tube of $G$ such
that $t$ is not a tube of $G-e,$ then let $t', t''$ be the tubes of $G-e$ such that $t'\cup t'' = t.$ Then:
$$
\Theta_e(t) =  \begin{cases}
\{t\},    &t \text{ a tube of } G-e \\
\{t',t''\},    & otherwise.
\end{cases}
$$
Now given a tubing $T$ on $G$ we define its image as follows:
  $$
 \Theta_e(T) = \bigcup_{t\in T} \Theta_e(t).
   $$
\end{defi}

See Figures~\ref{f:tonks},~\ref{f:slant} and~\ref{f:disc} for examples.
\begin{figure}[h]
                  \centering
                  \includegraphics[width=\textwidth]{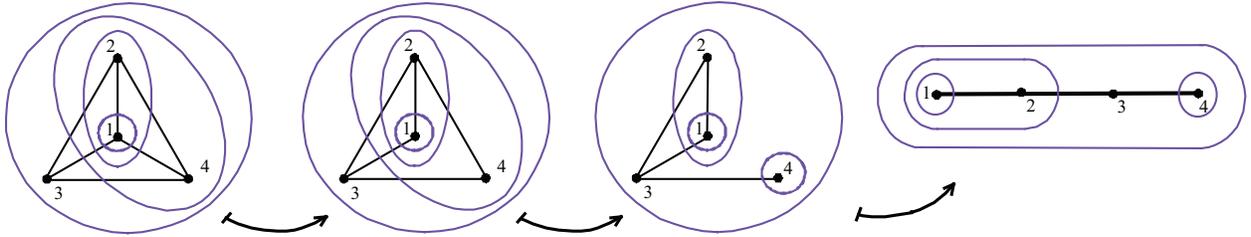}
                  \caption{The Tonks projection performed on $(1 2 4 3)$, factored by graphs. }\label{f:tonks}
  \end{figure}

\begin{lem}\label{l:newtonks}
For a graph $G$ with an edge $e$, $\Theta_e$ is a cellular surjection of polytopes $\K G \twoheadrightarrow
\K (G - e).$
\end{lem}
\begin{proof}
By Theorem~\ref{t:graph} we have that the  face posets of the
polytopes are isomorphic to the posets of tubings. The map takes a
tubing on $G$ to a tubing on $G-e$ with a greater or equal number of
tubes. This establishes its projective property.
 Also, for two tubings $U \prec U'$ of $G$ we see that either
$\Theta_e(U) = \Theta_e(U') $ or $\Theta_e(U) \prec \Theta_e(U').$ Thus the
poset structure is preserved by the projection.

Finally, the map is surjective, since given any tubing $T$ on $G-e$ we can find a (maximal) preimage $T'$ as
follows: First consider all the tubes of $T$ as a candidate tubing of $G.$ If it is a valid tubing, we have
our $T'.$ If not, then there must be a pair of tubes $t',t'' \in T$ which are adjacent via the edge $e,$ and
for which there are no tubes containing one of $t',t''$ but not the other. Then let $T_1$ be the result of
replacing that pair in $T$ with the single tube $t = t' \sqcup t''.$ If $T_1$ is a valid tubing of $G$ then
let $T'=T_1.$ If not, continue inductively.
\end{proof}

Composition of these cellular projections is commutative.
\begin{lem}\label{l:comm}
Let $e,e'$ be edges of $G.$ Then $\Theta_e \circ \Theta_{e'} =
\Theta_{e'} \circ \Theta_{e}.$
\end{lem}
\begin{proof}
Consider the image of a tubing  of $G$ under either composition. A
tube of $G$ that is a tube of both $G-e$ and $G-e'$ will persist in
the image. Otherwise it will be broken, perhaps twice. The same
smaller tubes will result regardless of the order of the breaking.
\end{proof}

By Lemma~\ref{l:comm} we can unambiguously use the following
notation:
\begin{defi}\label{d:bige}
 For any collection $E$ of edges of $G$ we denote by
$\Theta_E:\KG \twoheadrightarrow \K(G-E)$  the composition of projections $\{\Theta_e ~|~e \in E\}.$
\end{defi}
 Now the Tonks projection from leveled trees to trees can be
 described in terms of the tubings. Beginning with a tubing on the complete graph with numbered nodes, we achieve one on
 the path graph by deleting all the edges of the complete
 graph except for those connecting the nodes in consecutive order from 1 to $n$. See Figures~\ref{perm1}, \ref{perm2} and~\ref{f:tonks} for a
 picture to accompany
the following:
\begin{thm} \label{t:ftonks}
 Let $e_{i,k}$ be the edge between the nodes
$i,i+k$ of the complete graph $G$ on $n$ nodes. Let $P =
\{e_{i,k}~|~ i \in \{1,\dots,n-2\}~ and~ k \in \{ 2,\dots,n-i\}\}.$
These are all but the edges of the path graph.
  Then the following composition gives the
Tonks map:
$$
g^{-1}\circ \Theta_P \circ f = \Theta.
$$
\end{thm}
\begin{proof}
We begin with an ordered $j$-partition of $n$ drawn as a leveled tree $t$ with $n+1$ leaves numbered left to right,
 and $j$ levels
 numbered top to bottom.  The
 bijection $f$ tells us how to  draw $t$ as a tubing of the complete graph $K_n$ with numbered nodes.

 Consider the internal nodes of $t$ at level $i$.
In $f(t)$ there corresponds a tube $u_i$ containing the nodes of
$K_n$ with the same numbers as all the spaces between leaves of the
subtrees rooted at the internal nodes at level $i.$ The set in the
partition which is represented by the internal nodes at level $i$ is
the precise set of nodes of $u_i$ which are not contained by any
other tube in $f(t).$ The relative position of this subset of $[n]$
in our ordered partition is reflected by the relative nesting of the
tube $u_i$.

The map $\Theta_P$ has the following action on tubings. Let $u$ be a tube of $f(t).$
  We partition $u$ into subsets of consecutively numbered nodes such
 that no union of two subsets is consecutively numbered. Then the tubing $\Theta_P(f(t))$
contains the tubes given by these subsets.

Now we claim that the tree $g^{-1}(\Theta_P(f(t)))$ has
the same branching structure as $t$ itself. First, for any interior
node of $t$ there is a tube of consecutively numbered nodes of the
path graph, arising from the original tube of $f(t)$. This then
becomes a corresponding
 interior node of our tree,
 under the action of $g^{-1}$. Secondly, if any interior node of $t$ lies between the root of $t$
and a second interior node, then the same relation holds in the
tree between the corresponding interior nodes. This
follows from the fact that the action of any of our maps
$\Theta_{e_{i,k}}$ will preserve relative nesting. That is, for two
tubes $u \subset v$ we have that in the image of $\Theta_{e_{i,k}}$
any tube resulting from $u$ must lie within a tube resulting from
$v.$
\end{proof}

 To sum up,  there is a factorization of the Tonks cellular projection through various graph-associahedra. An example on an $n$-tubing is
 shown in Figure~\ref{f:tonks}, and another possible factorization of the projection in dimension 3 is demonstrated in Figure~\ref{f:slant}.

\begin{figure}[h]
\includegraphics[width=\textwidth]{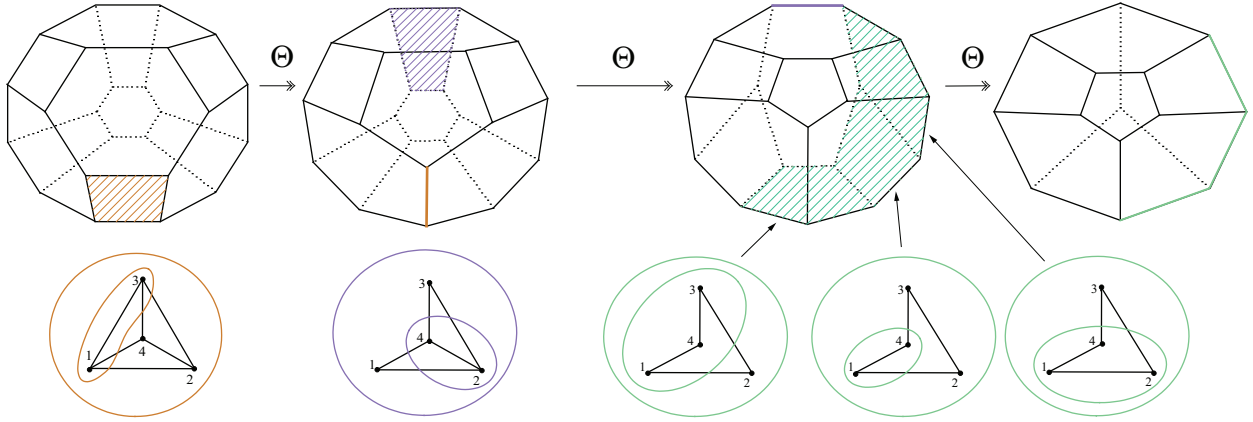}
\caption{A  factorization of the Tonks projection through 3 dimensional graph associahedra. The shaded facets
correspond to the shown tubings, and are collapsed as indicated to respective edges. The first, third and
fourth pictured polytopes are above views of $\P_4, \w_4$ and $\K_4$ respectively.} \label{f:slant}
\end{figure}

\subsection{Disconnected graph associahedra.}
The special case of extending the Tonks projection to graphs with
multiple connected components will be useful.
Consider a partition $S_1 \sqcup \dots \sqcup S_k$ of the
$n$ nodes of a connected graph $G$, chosen such that we have
connected induced subgraphs $G(S_i).$ Let $E_S$ be the set of edges
of G not included in any $G(S_i).$ Thus the graph $G-E_S$ formed by
deleting these edges will have the  $k$ connected components
$G(S_i).$
 In this situation
$\Theta_{E_S}$ will be a generalization of the Tonks projection to
the graph associahedron of a disconnected graph.

In Figure~\ref{f:disc} we show the extended Tonks projections in
dimension 2.

\begin{figure}[h]
                  \centering
                  \includegraphics[width=\textwidth]{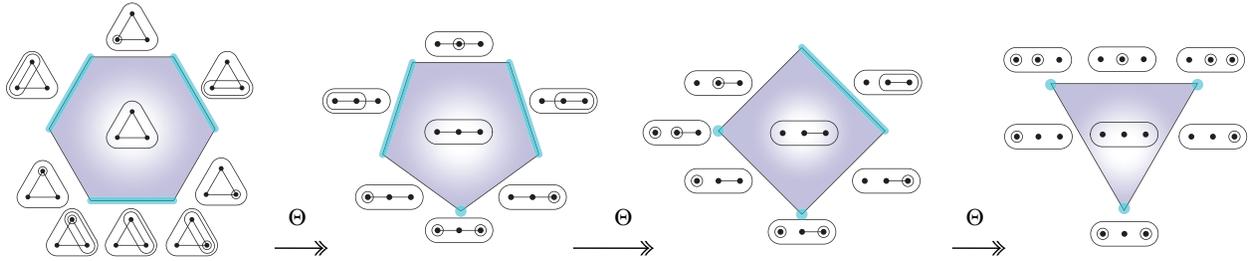}
                  \caption{Extending the Tonks projection to the 2-simplex.
 The highlighted edges are collapsed to the respective vertices. }\label{f:disc}
  \end{figure}

%
%

\section{Geometrical view of $\ssym$ and $\ysym$} \label{s:newview}
Before proving the graded algebra structures on graph associahedra
which our title promised, we motivate our point of view by showing
how it will fit with the well known graded algebra structures on
permutations and binary trees.
\subsection{Review of $\ssym$}
Let $\ssym$ be the graded vector space over $\Q$ with the $n^{th}$ component of its basis given by the
permutations $S_n.$ An element $\sigma \in S_n$ is given by its image $(\sigma(1),\dots,\sigma(n))$, often
without commas. We follow \cite{AS-MR} and \cite{AS-LR} and write $F_{u}$ for the basis element corresponding
to $u\in S_n$ and 1 for the basis element of degree 0. A graded Hopf algebra structure on $\ssym$ was
discovered by Malvenuto and Reutenauer  in \cite{MR}. First we review the product and coproduct and then show
a new way to picture those operations.

Recall that a permutation $\sigma$ is said to have a descent at
location $p$ if $\sigma(p) > \sigma(p+1).$ The $(p,q)$-shuffles of
$S_{p+q}$ are the $(p+q)$ permutations with at most one descent, at
position $p.$ We denote this set as $S^{(p,q)}.$
The product in $\ssym$ of two basis elements $F_u$ and $F_v$
for $u\in S_p$ and $v\in S_q$
is found by summing a term for each shuffle, created by composing the juxtaposition of $u$ and $v$ with the
inverse shuffle:
\begin{large}
$$
F_u\cdot F_v = \sum_{\iota \in S^{(p,q)}} F_{(u\times v)\cdot \iota^{-1}}.
$$
Here $u\times v$ is the permutation $(u(1),\dots,u(p),v(1)+p,\dots,v(q)+p)$.
\end{large}
\subsection{Geometry of $\ssym$}

The algebraic structure of $\ssym$ can be linked explicitly to the
recursive geometric structure of the permutohedra. In $\ssym$ we may
view our operands (a pair of permutations) as a vertex of the
cartesian product of permutohedra $\P_p \times \P_q.$ Then their
product is the formal sum of images of that vertex under the
collection of inclusions of $\P_p \times \P_q$ as a facet of
$\P_{p+q}.$ An example is in Figure~\ref{perm3}, where the product
is shown along with corresponding pictures of the tubings on the
complete graphs. To make this geometric claim precise we use the
facet isomorphism $\hat{\rho}_t$ which exists by
Theorem~\ref{t:facet}.

\begin{thm}\label{t:geo_pro}
The product in $\ssym$ of basis elements $F_u$ and $F_v$ for $u\in
S_p$ and $v\in S_q$ may be written:
\begin{large}
$$
F_u\cdot F_v =
\sum_{\iota\in S^{(p,q)}} F_{\hat{\rho}_{\iota}(u,v)}
$$
\end{large}
where  $\hat{\rho}_{\iota}$ is shorthand for $\hat{\rho}_{\iota([p])}.$
\end{thm}
\begin{figure}[h]
                  \centering
                  \includegraphics[width=\textwidth]{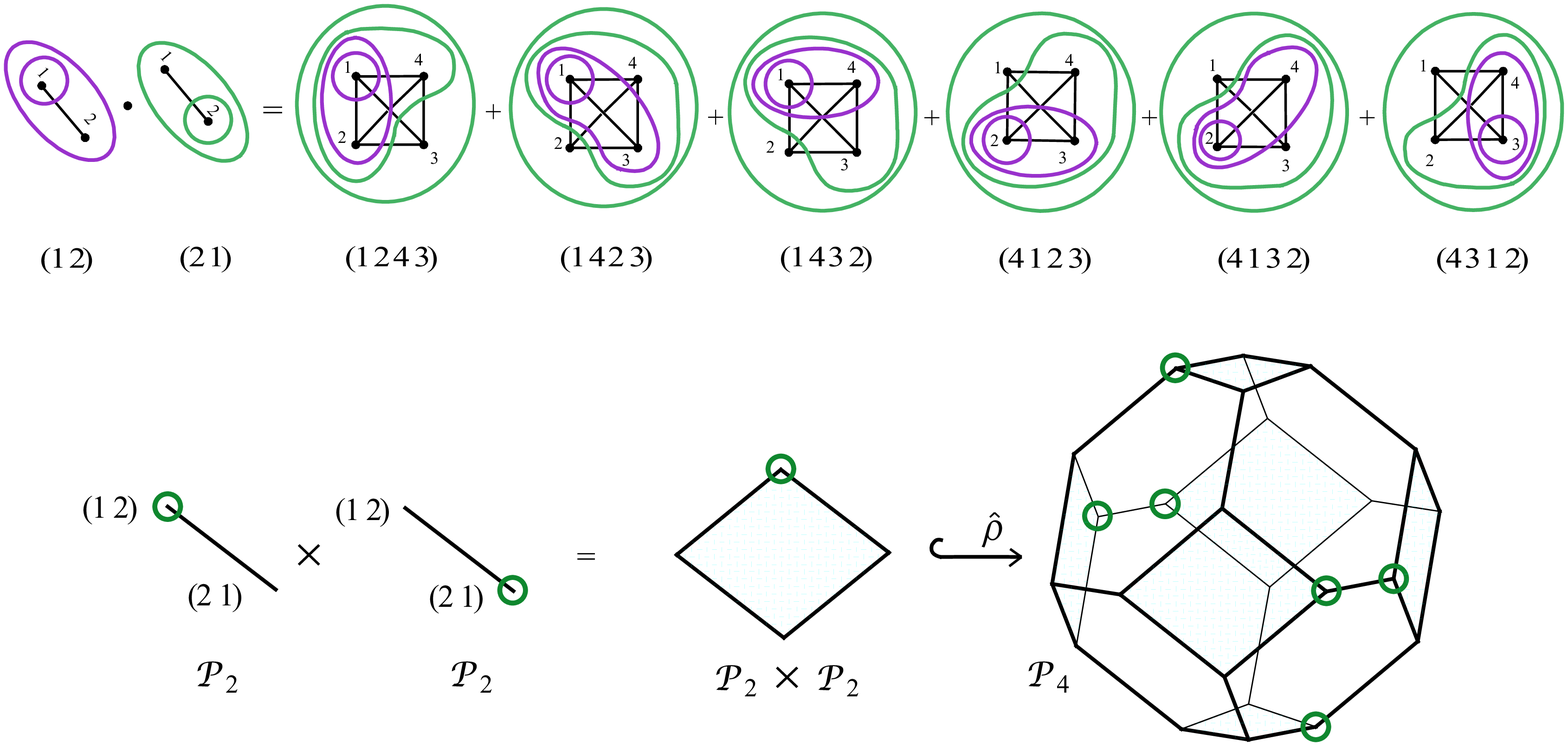}
                  \caption{The product in $\ssym$. Here (and similarly in all
                  our pictorial examples) we let the picture of the
                  graph tubing $u$ stand in for the basis element $F_u.$
                   In the picture of $\P_4$ the circled vertices from bottom to top are
                  the permutations in the product, as listed from left to right above.}\label{perm3}
  \end{figure}
\begin{proof}
 From Theorem~\ref{t:facet} we have that for each
tube $t$ of the graph, the corresponding facet of $\KG$ is
isomorphic to $\K\rec{t} \times \KG(t).$ In the case of the complete
graph $G = K_{p+q}$, for any tube $t$ of $p$ nodes, we have the
facet inclusion $\hat{\rho}_t:\P_q \times \P_p \to \P_{p+q}.$

We just need to review the definition of the isomorphism $\hat{\rho}$ from the proof of \cite[Theorem 2.9]{dev-carr},   and point out that
the permutation  associated to the tubing $\hat{\rho}_{\iota}(u,v)$ is indeed $(u\times v)\cdot\iota^{-1}$.

Given a shuffle $\iota$,  the $(p+q)$-tubing $\hat{\rho}_{\iota}(u,v)$ on $K_{p+q}$ is given
by including for each tube $t$ of $u$ the
 tube with nodes the
image $\iota(t).$  Denote the resulting set of tubes by
$\overline{\iota(u)}.$  For each tube $s$ of $v$ we include the tube
formed by $\iota([p])\cup\hat{\iota}(s),$ where $\hat{\iota}(i) = \iota(i+p).$ We denote the resulting set of
tubes as $\overline{\hat{\iota}(v)}.$

Now the tubing $\hat{\rho}_{\iota}(u,v)$ is defined to be
$\overline{\iota(u)}\cup\overline{\hat{\iota}(v)}.$ This tubing is precisely the complete tubing which
represents the permutation  $(u\times v)\cdot\iota^{-1}$.
\end{proof}

To sum up, in our view of $\ssym$ each permutation is pictured as a
tubing of a complete graph with numbered nodes.
 Since any subset of the nodes of
a complete graph spans a smaller complete graph, we can draw the terms of the
 product in $\ssym$ directly. Choosing a $(p,q)$-shuffle
is accomplished by choosing $p$ nodes of the $(p+q)$-node complete
graph $K_{p+q}$. First the permutation $u$ (as a $p$-tubing)
 is drawn upon the induced
 $p$-node
complete subgraph, according to the ascending order of the chosen nodes.
 Then the permutation $v$ is drawn upon the subgraph induced by the remaining $q$ nodes--with a caveat.
Each of the tubes of $v$ is expanded to also include the $p$ nodes
that were originally chosen.
 This perspective will generalize nicely
to the other graphs and their graph associahedra.

 In \cite{AS-MR} and \cite{AS-LR}
 the authors give a related geometric interpretation
of the products of $\ssym$ and $\ysym$ as expressed in the M\"obius basis.
 An interesting project for the future would be to apply that point of view to $\wsym.$


\begin{rema} \label{r:geo_cop}
  The coproduct of $\ssym$ can also be described geometrically in terms of the graph tubings.
The coproduct is usually described as a sum of all the ways of
splitting a permutation $u\in S_n$ into two permutations $u_i \in
S_i$ and $u_{n-i} \in S_{n-i}$:
\begin{large}
$$
\Delta (F_u) = \sum_{i=0}^{n}F_{u_i}\otimes F_{u_{n-i}}
$$
\end{large}
 where $u_i = (u(1)\dots u(i))$ and  $u_{n-i} = (u(i+1)-i\dots u(n)-i).$

Given an $n$-tubing $u$ of the
  complete graph on $n$ vertices we can find $u_i$ and $u_{n-i}$ just by restricting to the sub graphs (also complete) induced by
the nodes $1,\dots, i$ and $i+1,\dots, n$ respectively. For each tube $t\in u$ the two intersections of $t$ with the respective subgraphs
are included in the respective tubings $u_i$ and $u_{n-i}.$
An example is in Figure~\ref{permcop}.

Notice that this restriction of the tubings to subgraphs is the same as the result of performing the Tonks projection. Technically,
$(u_i,u_{n-i}) = \eta(\Theta_{E_i}(u)),$ where
$$E_i = \{e ~ an~edge ~ of ~G ~|~e \text{ connects a node $j\le i$ to a node $j' > i$}\}.$$
 and $\eta$ is from Lemma~\ref{l:comp}.
\end{rema}
\begin{figure}[h]
                  \centering
                  \includegraphics[width=\textwidth]{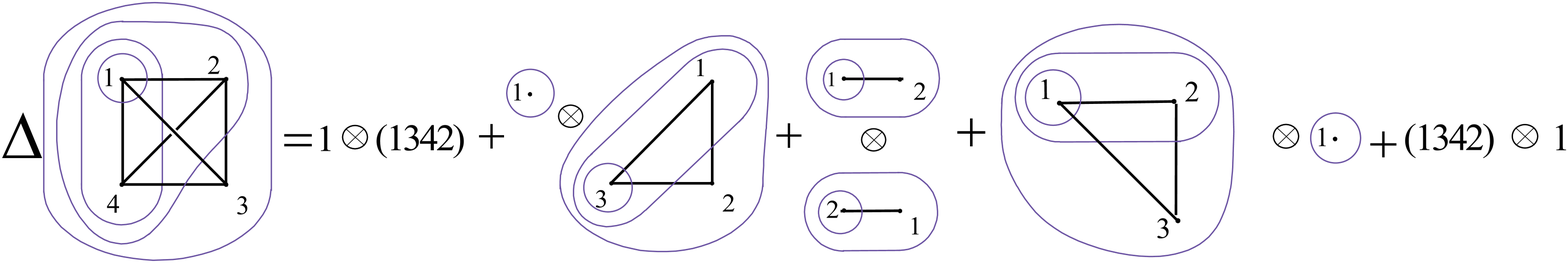}
                  \caption{The coproduct in $\ssym.$}\label{permcop}
                   $$\Delta F_{(1 3 4 2)}= 1\otimes F_{(1342)}+F_{(1)}\otimes F_{(231)}+F_{(12)}\otimes F_{(21)}+F_{(123)}\otimes F_{(1)}+F_{(1342)}\otimes1.$$
  \end{figure}

\subsection{Review of $\ysym$}
  The product and coproduct of  $\ysym$ are described by Aguiar and Sottile in terms of
splitting and grafting binary trees \cite{AS-LR}. We can vertically split a tree into
 smaller trees at each leaf from top to bottom--as if
a lightning strike
    hits a leaf and splits the tree along the path to the root. We graft trees by attaching
    roots to leaves (without creating a new interior node at the graft.) The product
     of two trees with $n$ and $m$ interior nodes (in that order)
     is a sum of
    ${{n+m} \choose n}$ terms, each with $n+m$ interior nodes.
    Each is achieved by vertically splitting the first tree into
     $m+1$ smaller trees and then grafting them to the second tree, left to right.
A picture is in Figure~\ref{map24a}.

\subsection{Geometry of $\ysym$}\label{s:geomy}

Since Loday and Ronco demonstrated that the Tonks projection, restricted to vertices, gives rise to an
algebra homomorphism $\bs{\tau}:\ssym\to\ysym,$ it is no surprise that the processes of splitting and
grafting have geometric interpretations. Grafting corresponds to certain face inclusions of associahedra, and
splitting corresponds to the extension of Tonks's projection to disconnected graphs. To see the latter, note
that splitting at leaf $i$ is the same as deleting the edge from node $i$ to node $i+1.$ Thus the product in
$\ysym$ can be described with the language of path graph tubings, using a combination of facet inclusion and
the extended Tonks projection.  An example is shown in Figure~\ref{map24a}.

 Let $U$ be the path graph with $p$-tubing $u$
and $V$ the path graph with $q$-tubing $v.$ Given a shuffle $\iota\in S^{(p,q)}$ we can
 partition the nodes of $U$ into
the preimages $s_1,\dots,s_k$ of the connected components
$\iota(s_1),\dots,\iota(s_k)$ of the possibly disconnected subgraph
induced by the nodes $\iota([p])$ on the $(p+q)$-path. Let
$E_{\iota}$ be the edges of $U$ not included in the subgraphs
$U(s_i)$ induced by our partition. For short we denote the extended
Tonks projection $\Theta_{E_{\iota}}$ as simply $\Theta_{\iota}.$
Now $\Theta_{\iota}(u)$ is the projection of $u$ onto the possibly
disconnected graph  $U-E_{\iota}.$
 Recall from Lemma~\ref{l:comp} that a vertex of
$\K(U-E_{\iota})$ may be mapped to a vertex of $\K
U(s_1)\times\dots\times\K U(s_k).$ We call this map $\eta_{\iota}.$

\begin{thm}\label{r:geo}
The product in $\ysym$ can be written:
\begin{large}
$$
F_u\cdot F_v = \sum_{\iota \in S^{(p,q)}}F_{\hat{\rho}_{\iota}(\eta_{\iota}(\Theta_{\iota}(u)),v)}.
$$
\end{large}
 Where  $\hat{\rho}_{\iota}$ is shorthand notation for the isomorphism $\hat{\rho}_{\iota(s_1)\dots
\iota(s_k)}$ from Corollary~\ref{c:mfacet}.
\end{thm}
\begin{figure}[h]
                 \centering
              \includegraphics[width=\textwidth]{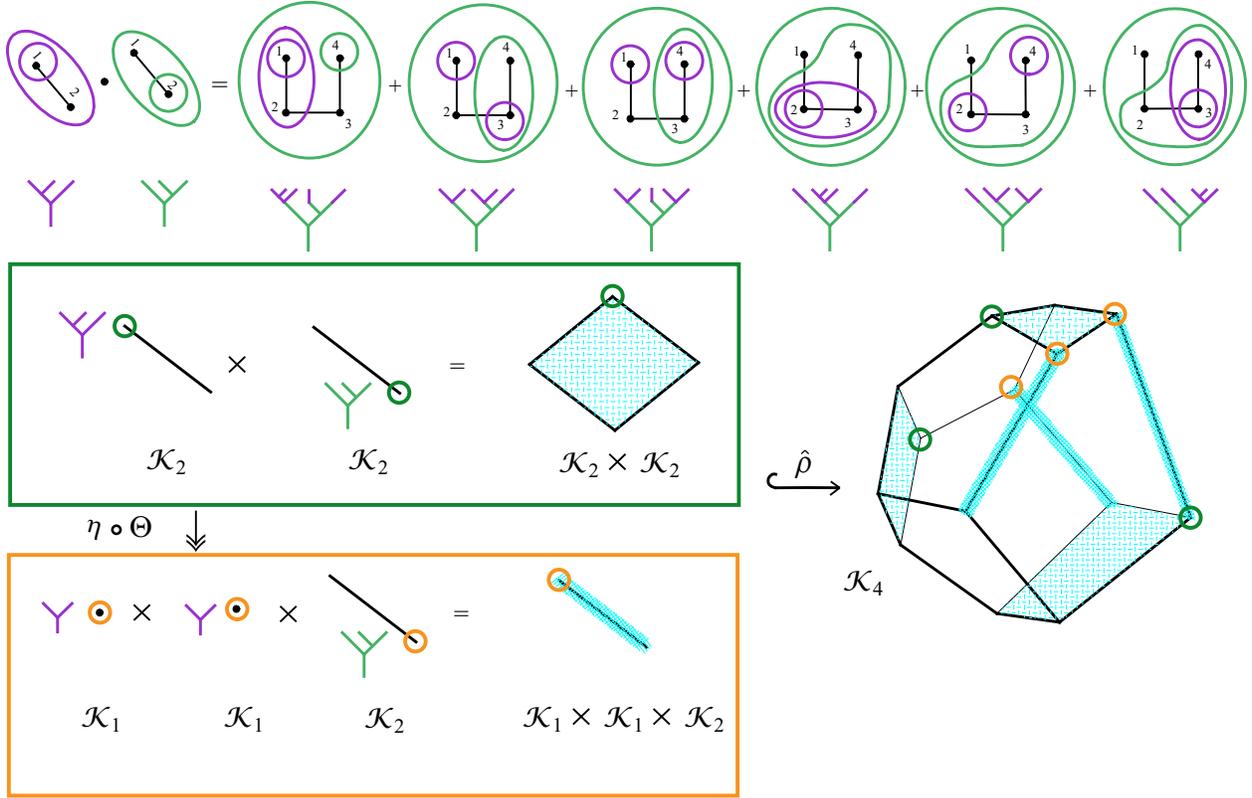}
                \caption{The product in $\ysym$. The circled vertices
                of $\K_4$ which are at the upper end of highlighted
                edges are the fifth, third and second terms of the
                product, in that order respectively from bottom to top in the
                picture.
                 }\label{map24a}
\end{figure}
\begin{proof}
We will explain how the splitting and grafting of trees in a term of the product
 may be put into the language of tubings, and then argue that the term thus described
is indeed the image of projections and inclusions as claimed.
 The product in $\ysym$ of two basis elements $F_u$ and $F_v$
for $u\in \y_p$ and $v\in \y_q$ is found by summing a term for each
shuffle $\iota \in S^{(p,q)}.$ We draw $u$
and $v$ in the form of tubings on path graphs of $p$ and $q$ nodes,
respectively.

Here is the non-technical description: first the $p$-tubing $u$
 is drawn upon the induced
subgraph of the nodes $\iota([p])$ according to the ascending order of the chosen nodes. However, each tube
may need to be first broken into several sub-tubes.
 Then the $q$-tubing $v$ is drawn upon the subgraph induced by the remaining $q$ nodes. In this last step,
each of the tubes of $v$ is expanded to also include any of the previously drawn tubes that its nodes are adjacent to.

To be precise, we first choose a shuffle
 $\iota \in S^{(p,q)}.$ Let $\hat{\iota}(i) = \iota(i+p).$ Our term of  $F_u\cdot F_v$ is the
$(p+q)$-tubing on the $(p+q)$-path given by the following:

First for each tube
$t \in u$ we include in our new tubing the tubes which are the connected components of the subgraph induced by the nodes $\iota(t).$
 Let $\overline{\iota(u)}$
denote the tubing constructed thus far. After this step in terms of trees, we have performed the
 splitting and chosen where to graft the (non-trivial) subtrees. In terms of trees the splitting occurs at leaves
labeled by $\hat\iota(i)-i$  for $i\in [q].$

 Second, for each tube $s \in v$ we include in our term of the product
 the tube formed by
$$\hat{\iota}(s) \cup \{j \in t' \in \overline{\iota(u)} ~|~ t'
\text{ is adjacent to } \hat{\iota}(s) \}.$$ Let
$\overline{\hat{\iota}(v)}$ denote the tubes added in this second
step. Now we have completed the grafting operation.

Now we just point out that
 $\hat{\rho}_{\iota}(\eta_{\iota}(\Theta_{\iota}(u)) ,v)$  is precisely
the same as $\overline{\iota(u)}\cup\overline{\hat{\iota}(v)}$ . The
splitting of tubes of $u$ is accomplished by $\Theta_{\iota}$; and
$\eta_{\iota}$ simply recasts the result as an element of the
appropriate cartesian product. Then $\hat{\rho}_{\iota}$, as defined
in \cite{dev-carr}, performs the inclusion of that element paired
together with $v.$
\end{proof}

To summarize, the product in
$\ysym$ can be seen as a process of splitting and grafting, or equivalently of projecting and including.
From the latter viewpoint, we see the product of two associahedra vertices being achieved by projecting the first onto
a cartesian product of smaller associahedra, and then mapping that result paired with the second operand
into a large associahedron
via face inclusion. Notice that the reason
 the second path graph tubing $v$ can be input here is that any reconnected complement of a path graph
is another path graph.

\begin{rema}\label{r:y_cop}
The coproduct of $\ysym$ can also be described geometrically in
terms of the graph tubings. The coproduct is usually described as a
sum of all the ways of splitting a binary tree $u\in \y_n$ along
leaf $i$ into two trees: $u_i \in \y_i$ and $u_{n-i} \in \y_{n-i}$:
\begin{large}
$$
\Delta (F_u) = \sum_{i=0}^{n}F_{u_i}\otimes F_{u_{n-i}}.
$$
\end{large}

 Given an $n$-tubing $u$ of the
  path graph on $n$ vertices we can find $u_i$ and $u_{n-i}$ just by restricting to the sub graphs (also paths) induced by
the nodes $1,\dots, i$ and $i+1,\dots, n$ respectively. For each tube $t\in u$ the two intersections of $t$ with the respective subgraphs
are included in the respective tubings $u_i$ and $u_{n-i}.$
An example is in Figure~\ref{assocop}.

Notice that this restriction of the tubings to subgraphs is the same as the result of performing the extended Tonks projection, just as described
in Remark~\ref{r:geo_cop}.
\end{rema}

\begin{figure}[h]
                  \centering
                  \includegraphics[width=\textwidth]{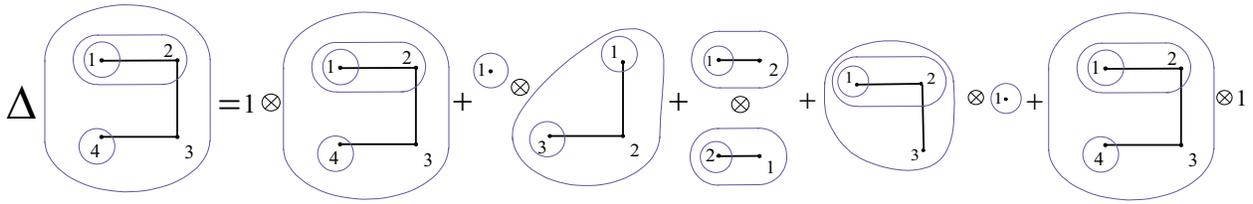}
                  \caption{The coproduct in $\ysym.$ }\label{assocop}
  \end{figure}

%
%

\section{The algebra of the vertices of cyclohedra}\label{s:cyclo}
Recall that the $(n-1)$-dimensional cyclohedron $\mathcal{W}_n$ has
vertices which are indexed by $n$-tubings on the cycle graph of
$n$-nodes. We will define a graded algebra with a basis which
corresponds to the vertices of the cyclohedra, and whose grading
respects the dimension (plus one) of the cyclohedra.
\begin{defi}
Let $\wsym$ be the graded vector space over $\Q$ with the $n^{th}$
component of its basis given by the $n$-tubings on the cycle graph
of $n$ cyclically numbered nodes. By $W_n$ we denote the set of
$n$-tubings on the cycle graph $C_n.$ We write $F_{u}$ for the basis
element corresponding to $u\in W_n$ and 1 for the basis element of
degree 0.
\end{defi}

\subsection{Graded algebra structure}
Now we demonstrate a product which respects the grading on $\wsym$
by following the example described above for $\ssym.$
 The product in $\wsym$ of two basis elements $F_u$ and $F_v$
for $u\in W_p$ and $v\in W_q$ is found by summing a term for each
shuffle $\iota\in S^{(p,q)}.$ First the
$p$-tubing $u$
 is drawn upon the induced
subgraph of the nodes $\iota([p])$ according to the ascending order of the chosen nodes. However, each tube
may need to be first broken into several sub-tubes, since the induced graph on the nodes $\iota([p])$ may
have connected components $\iota(s_1),\dots,\iota(s_k)$ (as described in Section~\ref{s:geomy}).
 Then the $q$-tubing $v$ is drawn upon the subgraph induced by the remaining $q$ nodes. However,
each of the tubes of $v$ is expanded to also include any of the previously drawn tubes that its nodes are adjacent to.

\begin{defi}\label{d:cyc_prod}
\begin{large}
$$
F_u\cdot F_v = \sum_{\iota \in S^{(p,q)}}F_{\hat{\rho}_{\iota}(\eta_{\iota}(\Theta_{\iota}(u)),v)}.
$$
\end{large}

Where  $\hat{\rho}_{\iota}$ is shorthand notation for the isomorphism $\hat{\rho}_{\iota(s_1)\dots
\iota(s_k)}$ from Corollary~\ref{c:mfacet}. Also 1 is a two-sided unit.
\end{defi}
An example is shown in Figure~\ref{permcyc}. For an example of finding a term in the product given a
specific shuffle $\iota,$ see the second part of Figure~\ref{f:hom}.

\begin{thm}\label{t:biggy}
The product we have just described makes $\wsym$ into an associative
graded algebra.
\end{thm}
\begin{proof}
Given a shuffle $\iota$ we will use the fact that
$$\hat{\rho}_{\iota}(\eta_{\iota}(\Theta_{\iota}(u)) ,v) = \overline{\iota(u)}\cup\overline{\hat{\iota}(v)}$$
 as defined in the proof of Theorem~\ref{r:geo}.
First we must check that the result of the product is indeed a sum
of valid $(p+q)$-tubings of the cycle graph $C_{p+q}.$ We claim that
in each term the new tubes we have created are pairwise compatible.
This being shown, we will be able to deduce that since $p+q$ tubes
were used in the construction, then the resulting term will
necessarily have $p+q$ tubes as well. To check our claim we compare
pairs of tubes in one or both of $\overline{\iota(u)}$ and
$\overline{\hat{\iota}(v)}.$ There are six cases:
\begin{enumerate}
\item
By our method of construction, any tube of $\overline{\iota(u)}$ is
either nested within some of or
 far apart from all of
the tubes of  $\overline{\hat{\iota}(v)}.$
\item
If two tubes of $\overline{\iota(u)}$ are made up of nodes in $\iota(t)$ and $\iota(t')$ respectively for
nested tubes $t$ and $t'$ of $u$ then they will be similarly nested.
\item
 Two tubes from $\overline{\iota(u)}$ might both be made up of nodes in $\iota(t)$
for a single tube $t$ of $u.$ In that case they are guaranteed to be
far apart, since their respective nodes together cannot be a
consecutive string (mod $p+q$).
\item
If two tubes of $\overline{\iota(u)}$ are made up of nodes in
$\iota(t)$ and $\iota(t')$ respectively for far apart tubes $t$ and
$t'$ of $u,$ then we claim they will be far apart. This is true
since if two nodes $a$ and $b$ are nonadjacent in the cycle graph
$C_p$ then the two nodes $\iota(a)$ and $\iota(b)$ will be
nonadjacent in $C_{p+q}.$
\item
If two tubes of $\overline{\hat{\iota}(v)}$ contain some nodes in $\hat{\iota}(t)$ and $\hat{\iota}(t')$
respectively for nested tubes $t$ and $t'$ of $u$ then they will be similarly nested. This follows from the
fact that $\iota(t)$ will only be adjacent to nodes that $\iota(t')$ is also adjacent to.
\item
Finally, if two tubes of $\overline{\hat{\iota}(v)}$ contain some
nodes in $\hat{\iota}(t)$ and $\hat{\iota}(t')$ respectively for far
apart tubes $t$ and $t'$ of $u,$ then we claim they will be far
apart. This final case depends on a special property which the cycle
graphs exemplify.
 Given any subset of $k$ of the nodes of $C_n$, the reconnected complement
of that subset is the cycle graph $C_{n-k}.$ Specifically the
reconnected complement of $C_{p+q}$ with respect to the nodes
$\iota(p)$ is the graph $C_q.$ Thus even the expanded tubes of
$\overline{\hat{\iota}(v)}$  remain far apart as long as their
components from $\hat{\iota}(q)$ were far apart; and this last
property is guaranteed since $\iota$ preserves the cyclic order.
\end{enumerate}
Thus we have shown that the result of multiplying two basis elements
is again a basis element, and that this multiplication respects the
grading. That this multiplication is associative is a corollary of
the following result regarding how the multiplication is preserved
under a map from $\ssym,$  specifically a corollary of
Theorem~\ref{t:permhom}.
\end{proof}

\begin{figure}[h]
                  \centering
                  \includegraphics[width=\textwidth]{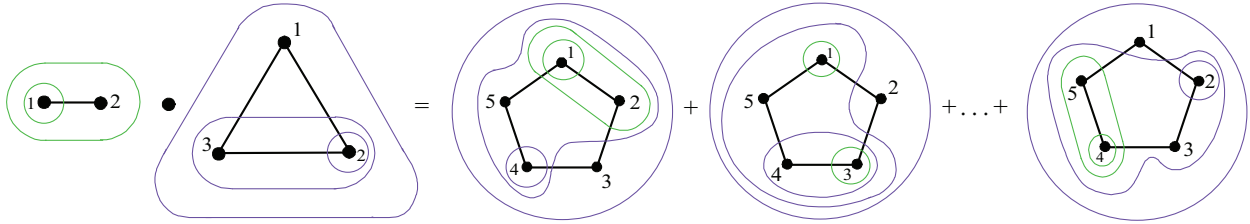}
                  \caption{A product of cycle graph tubings (10 terms total).  }\label{permcyc}
  \end{figure}
\begin{rema}
Note that the cases (1) and (2) are explainable simply by the fact
that we start with two valid tubings and multiply them in the given
order. Cases (3)-(6) however can be jointly explained based upon the
fact that given any subset of $k$ of the nodes of $C_n$, the
reconnected complement of that subset is the cycle graph $C_{n-k}.$
Cases (3)-(5) specifically rely on the fact that the reconnected
complement of $C_{p+q}$ with respect to the nodes $\hat{\iota}(q)$
is the graph $C_p.$ This property is true of many other graph
sequences, including the complete graphs and the path graphs.  The
property is also more simply stated as follows: the reconnected
complement of $G_i$ with respect to any single node is $G_{i-1}.$
\end{rema}
\begin{rema}\label{r:geo_cyc}
Once again we can interpret the product geometrically. The entire
contents of the proof of Theorem~\ref{r:geo} apply here, with the term ``path''
everywhere replaced with the term ``cycle.'' Thus a term in the product can be
seen as first projecting the cyclohedron
 vertex $u$
onto a collection of sub-path-graphs of the cycle. We then map the
vertex of this cartesian product of associahedra, paired with the
second vertex of the cyclohedron represented by $v,$ into the large
cyclohedron via the indicated face inclusion. The usual picture is
in Figure~\ref{f:demo_cyc}.
\end{rema}

\begin{figure}[h]
                  \centering
                  \includegraphics[width=\textwidth]{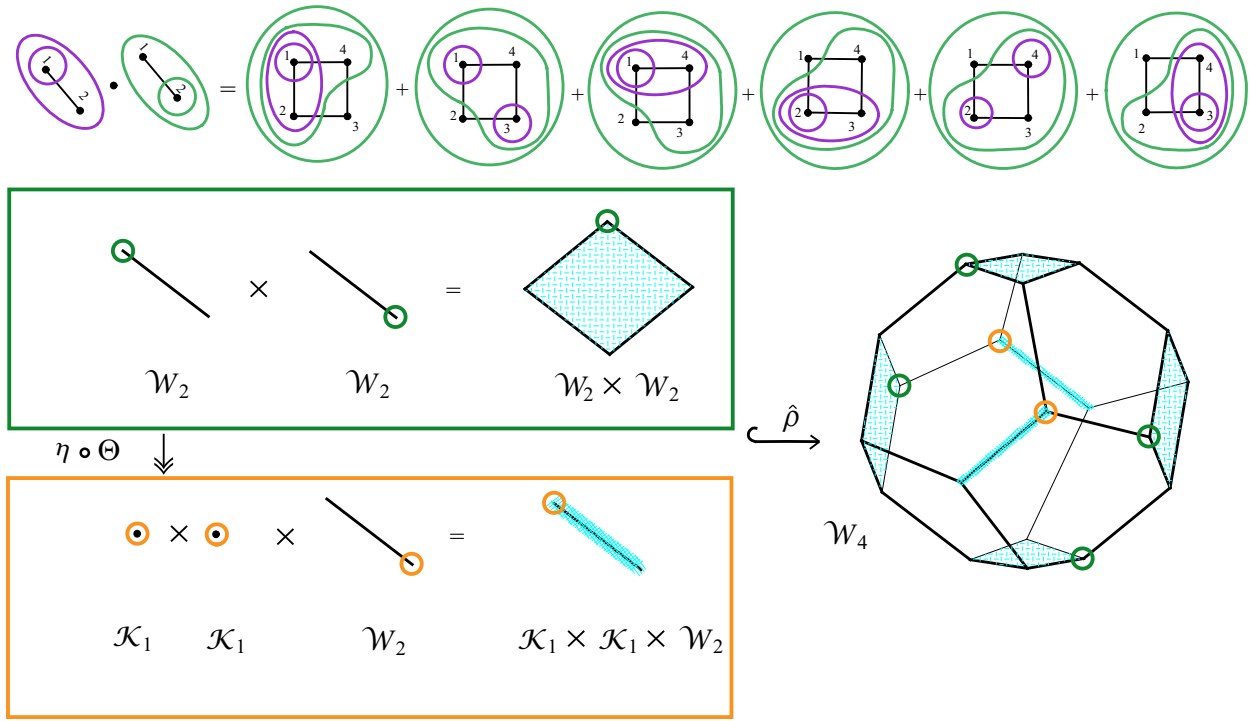}
                  \caption{The product in
                  $\wsym$. The second and fifth terms of the product are the
ones that use a Tonks projection; they are found as the two vertices at the tips of
included (highlighted) edges.}\label{f:demo_cyc}
  \end{figure}

\subsection{Algebra homomorphisms}

Next we consider the map from the permutohedra to the cyclohedra
which is described via the deletion of edges--from the complete
graphs to the cycle graphs--and point out that this is an algebra
homomorphism. Recall from Theorem~\ref{t:ftonks} that $\Theta_P$ is
the Tonks projection viewed from the graph associahedra point of
view; via deletion of the edges of the complete graph except for the
path connecting the numbered nodes in order. Let $\Theta_c$ be the
map defined just as $\Theta_P$ but without deleting the edge from
node $n$ to node $1.$ Thus we will be deleting all the edges except
those making up the cycle of numbered nodes in cyclic order. Define
 a map from $\ssym$ to $\wsym$ on basis elements by:
$$\hat{\Theta}_c(F_u) = F_{\Theta_c(u)}.$$

\begin{thm}\label{t:permhom}
The map $\hat{\Theta}_c$ is an algebra homomorphism from $\ssym$ onto
$\wsym.$
\end{thm}
\begin{proof}
 For $u\in S_p$ and $v\in
S_q$ we compare $\hat{\Theta}_c(F_u\cdot F_v)$ with $\hat{\Theta}_c(F_u)\cdot
\hat{\Theta}_c(F_v).$ Each of the multiplications results in a sum
of ${p+q \choose p}$ terms. It turns out that comparing the results
of the two operations can be done piecewise. Thus we check that for
a given shuffle $\iota\in S^{(p,q)}$ the respective terms of our
two operations agree: we claim that
$$
\Theta_c(\hat{\rho}_{\iota}(\eta_{\iota}(\Theta_{\iota}(u)) ,v)) =
\hat{\rho}_{\iota}(\eta_{\iota}(\Theta_{\iota}(\Theta_c(u))) ,\Theta_c(v))
$$
or equivalently:
$$\Theta_c(\overline{\iota(u)} \cup  \overline{\hat{\iota}(v)}) = \overline{\iota(\Theta_c(u))} \cup  \overline{\hat{\iota}(\Theta_c(v))}.$$
Here $\overline{\iota(u)}$ and $\overline{\hat{\iota}(v)}$ are as
described in the proof of Theorem~\ref{r:geo}, and the righthand
side of the equation is using the notation of
Definition~\ref{d:cyc_prod}.
 The justification is a straightforward comparison of the
indicated operations on individual tubes. There are two cases:
\begin{enumerate} \item For $t$ a
tube of $u$, the right-hand side first breaks $t$ into several
smaller tubes by deleting certain edges of the $p$-node complete
graph, then takes each of these to their image under $\iota,$
breaking them again whenever they are no longer connected. The
left-hand side takes $t$ to $\iota(t),$ a tube of the complete graph
on $p+q$ nodes, and then breaks $\iota(t)$ into the tubes that
result from deleting the specified edges of the $(p+q)$-node
complete graph. In this last step, by Lemma~\ref{l:comm}, we can
delete first those edges that also happen to lie in the complete
subgraph induced by $\iota([p]),$ and then the remaining specified
edges. Thus we have duplicated the left-hand side and get the same
final set of tubes on either side of the equation.
\item For $s$ a tube of $v$,
the right-hand side first breaks $s$ into several smaller tubes by
deleting certain edges of the $q$-node complete graph, then takes
each of these to their image under $\hat{\iota},$ expanding them to
include tubes of $\overline{\iota(\Theta_c(u))}$. The left-hand side
takes $s$ to $\hat{\iota}(s)\cup\iota([p]),$ and then breaks the
result into the tubes that result from deleting the specified edges
of the $(p+q)$-node complete graph. Again we can delete first those
edges that also happen to lie in the complete subgraph induced by
$\iota([p]),$ and then the remaining specified edges, duplicating
the left-hand process.
\end{enumerate}
An illustration of the two sides of the equation is in
Figure~\ref{f:hom}.
The surjectivity of $\hat{\Theta}_c$  follows from the surjectivity of the
generalized Tonks projection, as shown in Lemma~\ref{l:newtonks}.
 \end{proof}
\begin{figure}[h]
                  \centering
                  \includegraphics[width=\textwidth]{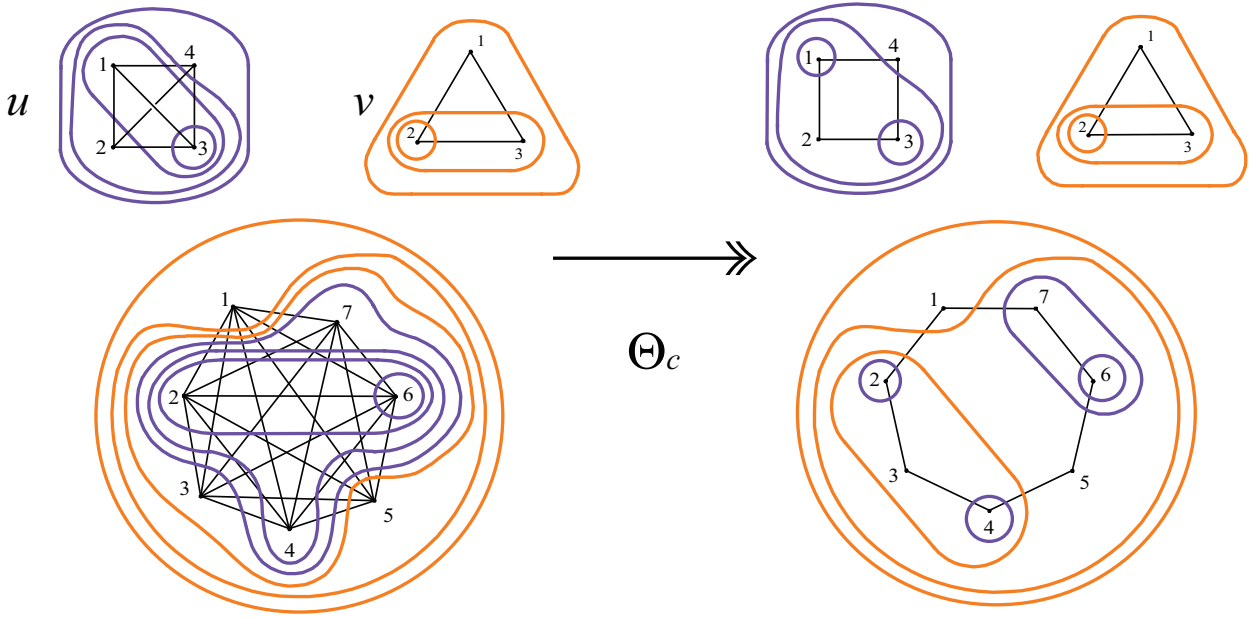}
                  \caption{An example of the equation
                  $\Theta_c(\overline{\iota(u)} \cup  \overline{\hat{\iota}(v)}) = \overline{\iota(\Theta_c(u))} \cup  \overline{\hat{\iota}(\Theta_c(v))}.$
                  where $u=(2314),$  $v=(312)$ and  $\iota([4]) = \{2,4,6,7\}.$ }\label{f:hom}
  \end{figure}

Let $C_n$ be the cycle graph on $n$ numbered nodes and let $w$ be
the edge from node $1$ to node $n.$ We can define a map from $\wsym$
to $\ysym$ by:
$$\hat{\Theta}_w(F_u) = F_{\Theta_w(u)},$$
 for $u\in
W_n.$


\begin{thm}\label{t:whom}
$\hat{\Theta}_w$ is a surjective homomorphism of graded algebras $\wsym \to
\ysym$.
\end{thm}

\begin{proof}
In \cite{LR} it is shown that the map we call $\bs{\tau}$ is an algebra homomorphism from $\ssym$ to $\ysym.$
In the previous theorem we demonstrated that the map $\hat{\Theta}_c$ is a surjective algebra homomorphism
from $\ssym$ to $\wsym.$ Now, since $\tau$ is the same map on vertices
 as our $\Theta_P$ (from Theorem~\ref{t:ftonks}), then
 the relationship of these three is:
$$\bs{\tau} = \hat{\Theta}_w \circ \hat{\Theta}_c.$$  Thus
$\hat{\Theta}_w$  is an algebra homomorphism from $\ssym$ onto
$\wsym.$
\end{proof}

\begin{rema}
The existence of a surjective algebra homomorphism from $\ssym$
 not only allows us a shortcut to demonstrating associativity, but leads to an alternate description of the product
in the range of that homomorphism.  This may be achieved in three steps:
     \begin{enumerate}
     \item Lifting our $p$ and $q$-tubings in $\wsym$ to any preimages of the generalized Tonks projection $\Theta_c$
 on the complete graphs on $p$ and $q$ nodes.
     \item Performing the product of these complete graph tubings in $\ssym$,
     \item and finding the image of the resulting terms under the homomorphism $\hat{\Theta}_c$.
     \end{enumerate}
This description is independent of our choices of preimages in $\ssym$ due to the homomorphism property.
\end{rema}
The following corollary follows directly from the properties of the surjective algebra homomorphisms of
Theorems~\ref{t:permhom} and~\ref{t:whom}.
\begin{cor}\label{c:cool}
$\wsym$ is a left $\ssym$-module under the definition $F_u\cdot F_{\Theta_c(v)} = \hat{\Theta}_c(F_u\cdot
F_v)$ and a right $\ssym$-module under the definition $F_{\Theta_c(u)}\cdot F_v = \hat{\Theta}_c(F_u\cdot
F_v).$ $\ysym$ is a left $\wsym$-module under the definition $F_u\cdot F_{\Theta_w(v)} =
\hat{\Theta}_w(F_u\cdot F_v)$ and a right $\wsym$-module under the definition $F_{\Theta_w(u)}\cdot F_v =
\hat{\Theta}_w(F_u\cdot F_v).$
\end{cor}
%
%

\section{Extension to the faces of the polytopes}\label{s:chap}
Chapoton was the first to point out the fact that in studying the
Hopf algebras based on vertices of permutohedra, associahedra and
cubes, one need not restrict their attention to just the
zero-dimensional faces of the polytopes. He has shown that the Loday
Ronco Hopf algebra $\ysym,$ the Hopf algebra of permutations
$\ssym,$ and the Hopf algebra of quasisymmetric functions $\qsym$
are each subalgebras of algebras based on the the trees, the ordered
partitions, and faces of the hypercubes respectively \cite{chap}.
 Furthermore, he has demonstrated
that these larger algebras of faces are bi-graded and possess a
differential. Here we point out that the cyclohedra based algebra
$\wsym$ can be extended to a larger algebra based on all the faces
of the cyclohedra as well, and conjecture the additional properties.

Chapoton's product structure on the permutohedra faces is given in
\cite{chap} in terms of ordered partitions. Let $\tilde{\ssym}$ be
the graded vector space over $\Q$ with the $n^{th}$ component of its
basis given by the ordered partitions of $[n].$ We write $F_{u}$ for
the basis element corresponding to the $m$-partition $u:[n] \to
[m],$ for $0\le m\le n,$ and 1 for the basis element of degree 0.
The product in $\tilde{\ssym}$ of two basis elements $F_u$ and $F_v$
for $u:[p]\to[k]$ and $v:[q]\to[l]$ is found by summing a term for
each shuffle, created by composing the juxtaposition of $u$ and $v$
with the inverse shuffle:
\begin{large}
$$
F_u\cdot F_v = \sum_{\sigma \in S^{(p,q)}} F_{(u\times v)\cdot \sigma^{-1}}.
$$
\end{large}
Here $u\times v$ is the ordered $(k+l)$-partition of $[p+q]$ given
by:
$$
(u\times v)(i) = \begin{cases}
u(i),    &i \in [p] \\
v(i-p)+k,    & i \in \{p+1,\dots,p+q\}.
\end{cases}
$$
The bijection between tubings of complete graphs and ordered partitions allows us to
write this product geometrically.
\begin{thm}\label{t:geochap}
The product in $\tilde{\ssym}$  may be written as:
\begin{large}
$$
F_u\cdot F_v = \sum_{\iota \in S^{(p,q)}} F_{\hat{\rho}_{\iota}(u,v)}
$$
\end{large}
which is just Definition~\ref{t:geo_pro} extended to all pairs of faces $(u,v).$
\end{thm}
\begin{proof}
Recall that we found the bijection between tubings of complete
graphs and ordered partitions by noting that each tube contains some
numbered nodes
   which are not contained in any other tube.
   These subsets of $[n],$ one for each tube, make up the partition, and the
   ordering of the partition is from innermost to outermost tube.
   Now the Carr-Devadoss isomorphism $\hat{\rho}$ is a
bijection of face posets.
   With this in mind, the same argument applies as in the proof of
   Theorem~\ref{t:geo_pro}.
\end{proof}

In other words, we view our operands (a pair of ordered partitions)
as a face of the cartesian product of permutohedra $\P_p \times
\P_q.$  Then the product is the formal sum of images of that face
under the collection of inclusions of $\P_p \times \P_q$ as a facet
of $\P_{p+q}.$ An example is in Figure~\ref{perm34}.

\begin{figure}[h]
                 \centering
              \includegraphics[width=\textwidth]{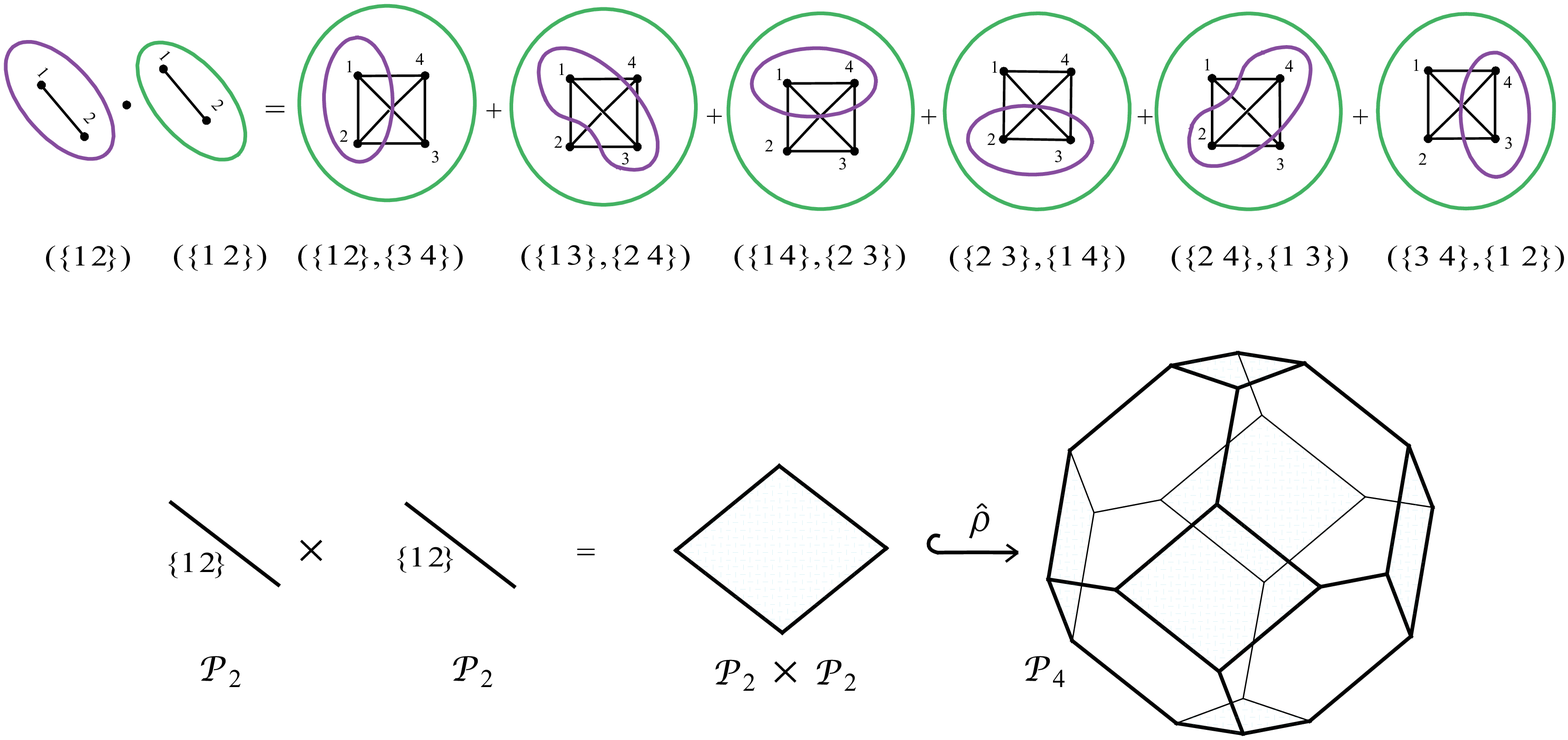}
                \caption{The product in $\tilde{\ssym}$. As an exercise the reader can
                delete edges in the graphs to illustrate the products in $\tilde{\wsym}$ and $\tilde{\ysym}.$ }\label{perm34}
\end{figure}

We leave to the reader the by now straightforward tasks of finding
the geometric interpretations of the coproduct on $\tilde{\ssym}$
and of the Hopf algebra structure on the faces of the associahedra,
$\tilde{\ysym}.$ Each one can be done simply by repeating earlier
definitions using tubings, but with all sizes of $k$-tubings as
operands.

Recall that the faces of the cyclohedra correspond to tubings of the cycle graph.
We will define a graded algebra with a basis which corresponds to
the faces of the cyclohedra, and whose grading respects the dimensions (plus one) of the cyclohedra.

\begin{defi}
Let $\tilde{\wsym}$  be the graded vector space over $\Q$ with the
$n^{th}$ component of its basis given by all the tubings on the
cycle graph of $n$ numbered nodes. By $\mathcal{W}_n=\K C_n$ we
denote the poset of tubings  on the cycle graph $C_n$ with a cyclic
numbering of nodes.  We write $F_{u}$ for the basis element
corresponding to $u\in \mathcal{W}_n$ and 1 for the basis element of
degree 0.
\end{defi}

The product in $\tilde{\wsym}$  of two basis elements $F_u$ and $F_v$ for $u\in \mathcal{W}_p$ and $v\in
\mathcal{W}_q$ is found by summing a term for each shuffle $\iota \in S^{(p,q)}.$  First the $l$-tubing $u$
 is drawn upon the induced
subgraph of the nodes $\iota([p])$ according to the ascending order of the chosen nodes. However, each tube
may need to be first broken into several sub-tubes.
 Then the $m$-tubing $v$ is drawn upon the subgraph induced by the remaining $q$ nodes. However,
each of the tubes of $v$ is expanded to also include any of the
previously drawn tubes that its nodes are adjacent to.

\begin{defi}
\begin{large}
$$
F_u\cdot F_v = \sum_{\iota \in S^{(p,q)}}F_{\hat{\rho}_{\iota}(\eta_{\iota}(\Theta_{\iota}(u)),v)}.
$$
\end{large} where the facet inclusion $\hat{\rho}_{\iota}$ is the same as in the two previous definitions
using this template: Definitions~\ref{d:sim_prod} and~\ref{d:cyc_prod}.
\end{defi}

\begin{thm}\label{t:wface}
The product we have just defined makes $\tilde{\wsym}$  into an
associative graded algebra.
\end{thm}
\begin{proof}
The proof is almost precisely the same as for the algebra of vertices of the cyclohedra, $\wsym.$
The only difference is that the tubings do not always have the maximum number of tubes.
First we must check that the result of the product is indeed a sum of valid tubings of the cycle graph $C_{p+q}.$
We claim that in each term the new tubes we have created are pairwise compatible.
 The cases to be checked and the reasoning for each are exactly as shown in the proof of Theorem~\ref{t:biggy}.
Associativity is shown by lifting the tubings to be multiplied to tubings on the complete graphs which are
preimages of the extended Tonks projection, and performing the multiplication in Chapoton's algebra. The fact
that the extended Tonks projection does preserve the structure here is again an easy check of corresponding
terms, following the same pattern as for the $n$-tubings in Theorem~\ref{t:permhom}.
\end{proof}

\begin{rema}
Chapoton has shown the existence of a differential graded structure on the algebras of faces of the
permutohedra, associahedra, and cubes \cite{chap}. The basic idea is simple, to define the differential as a
signed sum of the bounding sub-faces of a given face. Here we leave for future investigation the possibility
of extending this differential to algebras of graph associahedra.
\end{rema}

%
%

\section{Algebras of simplices}\label{s:simp}

We introduce a curious new graded algebra whose $n^{th}$ component has dimension $n.$ We denote it
$\dsym.$ In fact $\dsym$ may be thought of as a graded algebra whose basis is made up of all the standard
bases for Euclidean spaces.

The graph associahedron of the edgeless graph on $n$ vertices is the $(n-1)$-simplex $\Delta^{n-1}$  of
dimension $n-1.$  Thus the final range of our extension of the Tonks projection is the $(n-1)$-simplex (see
Figure~\ref{f:disc}.) The product of vertices of the permutohedra can be projected to a product of vertices
of the simplices. First we define this product by analogy to the previously defined products of tubings of
graphs.
 Then we will show the product to be associative via a
homomorphism from the algebra of permutations. Finally we will give a formula for the product using positive
integer coefficients and standard Euclidean basis elements.

\begin{defi}
Let $\Delta Sym$ be the graded vector space over $\Q$ with the $n^{th}$ component of its basis given by the
$n$-tubings on the edgeless graph of $n$ numbered nodes. By $D_n$ we denote the set of $n$-tubings  on the
edgeless graph. We write $F_{u}$ for the basis element corresponding to $u\in D_n$ and 1 for the basis
element of degree 0.
\end{defi}

\subsection{Graded algebra structure on vertices}
Now we demonstrate a product which respects the grading on $\dsym$ by following the example described above
for $\ssym.$
 The product in $\dsym$ of two basis elements $F_u$ and $F_v$
for $u\in D_p$ and $v\in D_q$ is found by summing a term for each shuffle $\iota \in S^{(p,q)}.$  Our term of
$F_u\cdot F_v$ will be a $p+q$ tubing of the edgeless graph on $p+q$ nodes. Its tubes will include all the
nodes numbered by $\iota([p]).$ We denote these by $\overline{\iota(u)}.$ In addition we include all the
tubes $\hat{\iota}(t)$ for non-universal $t\in v,$ and the universal tube. We denote these by
$\overline{\hat{\iota}(v)}.$ By inspection we see that in this case
 the union
$$\overline{\iota(u)} \cup \overline{\hat{\iota}(v)} = \hat{\rho}_{\iota}(\eta(u),v).$$
Thus we write:
\begin{defi}\label{d:sim_prod}
\begin{large}
$$
F_u\cdot F_v = \sum_{\iota \in S^{(p,q)}} F_{\hat{\rho}_{\iota}(\eta(u),v)}.
$$
\end{large}Also 1 is a two-sided unit.
\end{defi}
 An example is shown in Figure~\ref{f:sim}.

\begin{figure}[h]
                  \centering
                  \includegraphics[width=\textwidth]{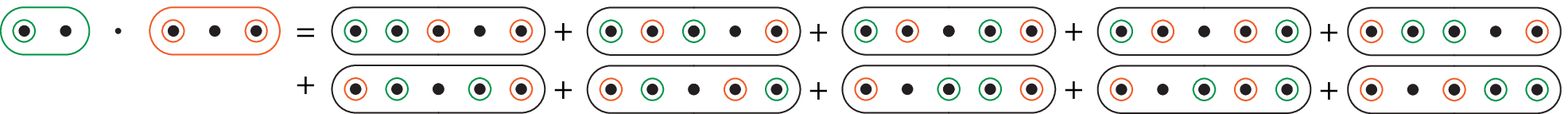}
                  \caption{The product in $\dsym.$}\label{f:sim}
  \end{figure}

\begin{thm}\label{t:sim}
The product we have just described makes $\dsym$ into an associative graded algebra.
\end{thm}

\begin{proof}
First we must check that the result of the product is indeed a sum of valid $(p+q)$-tubings of the
 edgeless graph on $p+q$ nodes.
This is clearly true, since there will always be only one node that is not a tube in any given term of the
product.
 Associativity is shown by the
existence of an algebra homomorphism $\ssym\to \dsym$
 which we will see next in Lemma~\ref{l:next}.
\end{proof}

\begin{defi}
Let $E_n = $ the set of edges of the complete graph on $n$ nodes. Deleting these gives a projection
$\Theta_n$ from the permutohedra to the simplices, as in Definition~\ref{d:bige}.  Then we define
$\hat{\Theta}_{\delta}:\ssym\to \dsym$ by
 $$\hat{\Theta}_{\delta}(F_u) = F_{\Theta_{n}(u)}$$ for $u$ an $n$-tubing of $K_n.$
\end{defi}

\begin{lem}\label{l:next}
The map $\hat{\Theta}_{\delta}$ is an algebra homomorphism.
\end{lem}
\begin{proof}
The proof follows precisely the same arguments as the proof for Theorem~\ref{t:permhom}, except that the
cases are simpler. We consider $\hat{\Theta}_{\delta}(F_u\cdot F_v)$ for $u\in S_p$ and $v\in S_q$. We
compare the result with $\hat{\Theta}_{\delta}(F_u)\cdot \hat{\Theta}_E(F_v).$ Each of the multiplications
results in a sum of ${p+q \choose p}$ terms. It turns out that comparing the results of the two operations
can be done piecewise. Thus we check that for a given shuffle $\iota$ the respective terms of our two
operations agree: we claim that
$$\Theta_{p+q}(\overline{\iota(u)} \cup  \overline{\hat{\iota}(v)}) = \overline{\iota(\Theta_p(u))} \cup  \overline{\hat{\iota}(\Theta_q(v))}.$$

On the left-hand side, $\Theta_{p+q}$ forgets all the information in the tubing it is applied to; except for
the following data: which of the nodes is the only node in the universal tube and not in any other tube. This
particular node is actually the node numbered by $j = \hat{\iota}(v^{-1}(q)).$ The effect of $\Theta_{p+q}$
is to create the
 $n$-tubing of the edgeless
graph with node $j$ as the only node that is not a tube.

 On the
right hand side $\Theta_p$ and $\Theta_q$  forget all but the value of $u^{-1}(p)$ and $v^{-1}(q)$
respectively. Then we create the tubing of the edgeless graph by first including all the nodes numbered by
$\iota([p])$ and then all the nodes except node $j=\hat{\iota}(v^{-1}(q)).$
\end{proof}
\begin{rema}
There are clearly algebra homomorphisms $\ysym\to\dsym$ and $\wsym\to\dsym$ described by the extended Tonks
projections from associahedra and cyclohedra to the simplices.
\end{rema}

\subsection{Formula for the product structure}
There is a simple bijection from $n$-tubings of the edgeless graph on $n$ nodes to standard basis elements of
$\Q^n.$ Let $e_m^n$ be the column vector of $\Q^n$ with all zero entries except for a 1 in the $m^{th}$
position. Associate $e_j^n = e_u$ with the $n$-tubing $u$ whose nodes are all tubes except for the $j^{th}$
node. Then use the product of $\dsym$ to define a product of two standard basis vectors of varying dimension:
$e_u\cdot e_v $ is the sum of all $e_w$ for $F_w$ a term in the product $F_u\cdot F_v.$ Then:
\begin{thm}
$$
e_j^p\cdot e_l^q = \sum_{i=l}^{p+l}{i-1 \choose l-1}{p+q-i \choose q-l}e_i^{p+q}
$$
\end{thm}
\begin{proof}
Let $e_l^q = e_v$ for the associated tubing $v.$  The only node not a tube of $v$ is node $l.$ We need only
keep track of where $l$ lands under $\hat{\iota}:[q]\to[p+q],$ where $\hat{\iota}$ is as in
Definition~\ref{d:sim_prod}.  The only possible images of $l$ are from $l$ to $p+l,$ thus the limits of the
summation. When $\hat{\iota}(l) = i$ there are several ways this could have occurred. $\hat{\iota}$ must have
mapped  $[l-1]$ to $[i-1]$ and the set $\{l+1,\dots,q\}$ to $\{i+1,\dots,p+q\}.$ The ways this can be done
are enumerated by the combinations in the sum.
 \end{proof}
 \begin{exam}
Consider the example product performed in Figure~\ref{f:sim}. Here the formula gives the observed quantities:
$$
e_2^2\cdot e_2^3 = 3e_2^5 + 4e_3^5 +3e_4^5.
$$
 \end{exam}

\subsection{Faces of the simplex}
 Note that a product of tubings on the edgeless graph
(with any number of tubes) is easily defined by direct analogy. Now we present a Hopf algebra with one-sided
unit based upon the face posets of simplices. Tubings on the edgeless graphs label all the faces of the
simplices. The number of faces of the $n$-simplex, including the null face and the $n$-dimensional face, is
$2^n.$ By adjoining the null face here we thus have a graded bialgebra with $n^{th}$ component of dimension
$2^n.$ It would be of interest to compare this with other algebras of similar dimension, such as $\qsym$.

\begin{defi}
Let $\tilde{\Delta Sym}$ be the graded vector space over $\Q$ with the $n^{th}$ component of its basis given
by the tubings on the edgeless graph of $n$ numbered nodes, with one extra basis element included in each component: corresponding to
the null facet $\emptyset_n$ is the collection consisting
 of all $n$ of the singleton tubes and the universal tube. By ${\mathcal D}_n$ we denote the set of
$n$-tubings  on the edgeless graph, together with $\emptyset_n$. We write $F_{u}$ for the basis element
corresponding to $u\in {\mathcal D}_n$ and 1 for the basis element of degree 0.
\end{defi}

Now we define the product and coproduct, with careful description of the units.

Let $u \in {\mathcal D}_p$ and $v\in {\mathcal D}_q.$
 For a given shuffle $\iota$ our term of $F_u\cdot F_v$ will be indexed by an element of ${\mathcal D}_{p+q}$ which
 will include all the nodes numbered by
$\iota([p]).$ In addition we include all the tubes $\hat{\iota}(t)$ for non-universal $t\in v,$ and the
universal tube. Thus the product is an extension of Definition~\ref{d:sim_prod}, with a redefined right
multiplication by the unit:

\begin{defi}\label{d:sim_prod2}
\begin{large}
$$
F_u\cdot F_v = \sum_{\iota \in S^{(p,q)}} F_{\hat{\rho}_{\iota}(\eta(u),v)} ;~~  1\cdot F_u = F_u ;~~
F_u\cdot 1 = F_{\emptyset_p}.
$$
\end{large}
\end{defi}
An example is shown in Figure~\ref{f:sim2}.
\begin{figure}[h]
                  \centering
                  \includegraphics[width=\textwidth]{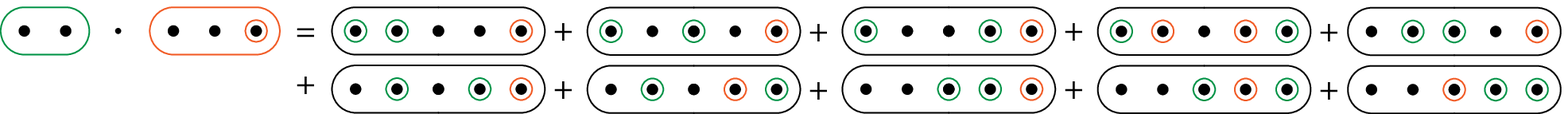}
                  \caption{The product in $\tilde{\dsym.}$}\label{f:sim2}
  \end{figure}

The coproduct is defined simply by restricting an element of ${\mathcal D}_n$ to
 its subgraphs induced by the nodes $1,\dots, i$ and $i+1,\dots n.$
 Given a tubing $u$ of the
  edgeless  graph on $n$ vertices we can find tubings $u_i$ and $u_{n-i}$ as follows: for each tube $t\in u$ we find the intersections of $t$ with
  the two sub-graphs (also edgeless) induced by
the nodes $1,\dots, i$ and $i+1,\dots, n$ respectively.

\begin{large}
$$
\Delta (F_u) = \sum_{i=0}^{n}F_{u_i}\otimes F_{u_{n-i}}.
$$
\end{large} An example is shown in Figure~\ref{f:sim_cop2}.
\begin{figure}[h]
                  \centering
                  \includegraphics[width=\textwidth]{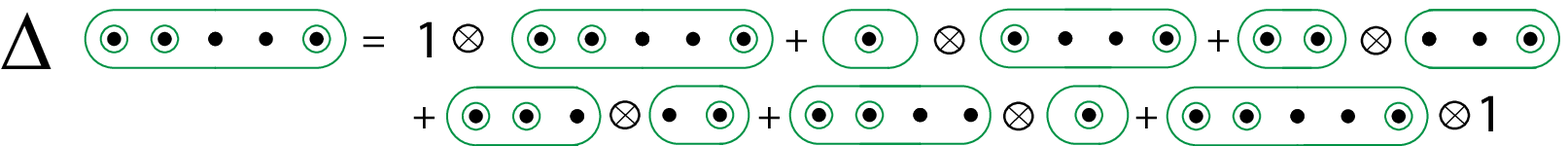}
                  \caption{The coproduct in $\tilde{\dsym.}$}\label{f:sim_cop2}
  \end{figure}

\begin{thm}\label{t:hopf}
The product and coproduct just defined form a graded bialgebra
structure on $\tilde{\Delta Sym}$, and therefore a (one-sided) Hopf
algebra.
\end{thm}
\begin{proof}
Associativity of the product is due to the observation that the result of multiplying a series of basis elements
only depends on the final operand. Coassociativity of the coproduct follows from the fact that terms of both
$(1\otimes \Delta)\Delta F_u$ and $(\Delta\otimes 1)\Delta F_u$ involve simply splitting the edgeless graph
into three consecutive pieces, and restricting $u$ to those pieces. Now we demonstrate that the product and
coproduct commute; that is, $\Delta(F_u\cdot F_v) = (\Delta F_u) \cdot (\Delta F_v).$ We describe a bijection
between the terms on the left-hand and right-hand sides, which turns out to be the identity map.
 Let $u \in {\mathcal D}_p$ and $v\in {\mathcal D}_q.$
  A term of the left-hand side depends first upon a $(p,q)$-shuffle $\iota$ to choose $p$ of the nodes of
  the $(p+q)$-node edgeless graph. Then after forming the
product, a choice of $i \in 0,\dots,p+q$ determines where to split the result into the pieces of the final
tensor product. Let $m= |\{x\in\iota([p]) ~:~ x \le i\}|.$ Thus $m$ counts the number of nodes in the image of $\iota$ which
lie before the split.

To create a term of the right hand side, we first split both $u$ and $v,$ then interchange and multiply the
resulting four terms as prescribed in the definition of the tensor product of algebras.
 Our matching term on the right-hand
side is the one formed by first splitting $u$ after node $m,$ and $v$ after node $i-m.$ Then
 in the first multiplication choose the $(m,i-m)$-shuffle $\sigma(x) =\iota(x)$, and in the
second use the $(q-m,q-i+m)$-shuffle  $\sigma'(x) = \iota(x+m)-i.$ These choices define an identity map from the set of terms of the left hand side to
those on the right.
\end{proof}
\begin{rema} $\tilde{\dsym}$ is closely related to the free associative \emph{trialgebra} on one variable described by Loday and Ronco in
\cite{LR3}. In \cite[Proposition 1.9]{LR3} the authors describe the products for that trialgebra.
Axiomatically, the first two products automatically form a \emph{dialgebra}.  The sum of these two products
appears as two of the terms of our shuffle-based product! We leave it to an interested reader to uncover the
precise relationship, perhaps duality, between the two structures.

\end{rema}

%
%
\bibliography{mybib}{}
\bibliographystyle{amsplain}

\end{document}